\documentclass[11pt,a4paper]{article}
\usepackage[T1]{fontenc}
\usepackage{lmodern,amsmath,amssymb,amsthm,mathtools,microtype,booktabs}
\usepackage[margin=27mm]{geometry}
\usepackage[colorlinks=true,linkcolor=blue,citecolor=blue,urlcolor=blue]{hyperref}
\hypersetup{pdftitle={From Hermite Chaos to an Attractive Log-Gas:
Explicit High-Rank Limits of Iterated-Logarithm and Tail Constants},
pdfauthor={Elina Moldavskaya}}
\newtheorem{theorem}{Theorem}[section]
\newtheorem{lemma}[theorem]{Lemma}
\newtheorem{proposition}[theorem]{Proposition}
\newtheorem{corollary}[theorem]{Corollary}
\theoremstyle{remark}\newtheorem{remark}[theorem]{Remark}
\newcommand{\dd}{\,\mathrm d}
\newcommand{\E}{\mathbb E}
\newcommand{\Var}{\operatorname{Var}}

\newcommand{\R}{\mathbb R}
\newcommand{\norm}[1]{\left\lVert#1\right\rVert}
\newcommand{\ip}[2]{\left\langle#1,#2\right\rangle}
\numberwithin{equation}{section}
\newcommand{\cA}{\mathcal A}
\newcommand{\cB}{\mathcal B}
\newcommand{\Fm}{\mathcal F_d}
\title{From Hermite Chaos to an Attractive Log-Gas:
Explicit High-Rank Limits of Iterated-Logarithm and Tail Constants}
\author{Elina Moldavskaya\\[4pt]
\small Technion--Israel Institute of Technology\\
\small\href{mailto:elina.mol@technion.ac.il}{\texttt{elina.mol@technion.ac.il}}}
\date{}
\begin{document}
\maketitle
\begin{abstract}
Let $Z$ be a standard Hermite process of order $m\ge2$ and
Hurst index $H\in(1/2,1)$. For $d<H-1/2$, let $Y$ be the
variance-normalized version of the weighted integral
$\int_0^1 x^{-d}\,\mathrm dZ(x)$. Its right tail satisfies
\[
 \log\Pr\{Y>t\}\sim
 -\frac12\left(\frac{t}{\Lambda_m}\right)^{2/m},
 \qquad t\to\infty,
\]
where the positive constant $\Lambda_m$ admits an
infinite-dimensional variational representation. The same constant
governs the sharp law of the iterated logarithm for weighted sums
of nonlinear functions of long-memory Gaussian sequences.
A general closed-form evaluation at finite order is not known.

As $m\to\infty$ with $H$ and $d$ fixed, the normalized constants
$\sqrt{m!}\,\Lambda_m$ decrease to an explicit positive function
of $H$ and $d$, with relative error $O(m^{-1})$. This function
also gives a lower bound for the normalized constant at every
finite order. An exact R\'enyi variational representation at each
order yields, in the high-rank limit, the entropy--energy
variational problem of an attractive logarithmic gas. The limiting
free energy is evaluated through the Selberg integral and the
Barnes $G$ function.

For $d\le0$, convergence remains uniform up to the memory boundary
$H=1/2$. After division by the square-root boundary factor
$\sqrt{2H-1}$, the high-rank and boundary limits commute. Their
common value is explicit; for constant weights it equals
$2^{3/2}\pi^3G(1/2)^8=2^{11/6}\pi e\,\mathsf A^{-12}$, where
$\mathsf A$ is the Glaisher--Kinkelin constant. The quadratic and cubic coefficients of the fixed-order expansion
of the normalized constant as $H\uparrow1$ converge, as $m\to\infty$,
to those of the explicit limiting function. Global finite-order bounds and a concentrating-weight
limit complement the main result.
\end{abstract}


\section{Introduction and main results}\label{sec:introduction}
For a Gaussian limit, variance normalization determines the constant in
the law of the iterated logarithm. For a limit in a higher Wiener chaos,
the constant depends on a nonlinear optimization problem. The same
optimization determines the leading logarithmic tail of the limiting
random variable. The purpose of this paper is to evaluate its limit as
the Hermite rank increases while the effective memory is kept fixed.

The probabilistic input is the weighted law of the iterated logarithm
in \cite{MoldavskayaLIL}. For a stationary Gaussian sequence $(X_t)$
with covariance exponent $\alpha$, a functional $\mathcal G$ of Hermite
rank $m$, and regularly varying weights $a_t=t^{-d}\ell_a(t)$, put
\[
 S_n=\sum_{t=1}^n a_t\mathcal G(X_t),\qquad A_n^2=\Var(S_n).
\]
Under the hypotheses recalled in Section~\ref{sec:model},
\[
 \limsup_{n\to\infty}
 \frac{|S_n|}{A_n(2\log\log n)^{m/2}}=\Lambda_m(\alpha,d)
 \quad\text{almost surely},
\]
where
\[
 m\ge2,\qquad 0<\alpha<1/m,\qquad 2d+m\alpha<1.
\]
Write $s=m\alpha$.
Equivalently $s=2-2H$, where $H=1-m\alpha/2$ is the Hurst index of the
limiting Hermite process. Thus fixing $s$ means fixing $H$; the
admissibility condition $2d+s<1$ reads $d<H-1/2$, and the memory
boundary $s\uparrow1$ is $H\downarrow1/2$, with $1-s=2H-1$.
The leading Hermite chaos has covariance exponent
$s$, and the variance has regularly varying power $2-2d-s$.
Fixing $(s,d)$ while increasing $m$ therefore preserves this power
and changes the order of the nonlinearity. The limit below concerns
deterministic constants from fixed-rank stochastic theorems. The tensor
definition in Section~\ref{sec:model} supplies an independent starting
point for all the deterministic results proved here; the stochastic
theorem identifies their meaning for weighted partial sums.

Set
\begin{equation}\label{lr:eq:Z}
 \eta=1-d,\qquad \mu_d(\dd x)=\eta x^{\eta-1}\dd x,\qquad
 Z_d(s)=\frac{2\eta^2B(\eta,1-s)}{2\eta-s}.
\end{equation}
\begin{theorem}[Explicit high-rank limit]\label{lr:thm:main}
Fix $d<1/2$ and $s\in(0,1)$ with $2d+s<1$. Put
\begin{equation}\label{lr:eq:Fgamma}
 \begin{split}
 \Fm(s)=\log \eta+\int_0^1\bigg[
 &\log\Gamma\left(\eta-\frac{st}{2}\right)
 +2\log\Gamma\left(1-\frac{st}{2}\right)\\
 &-\log\Gamma\left(\eta+1-\frac{s(1+t)}2\right)\bigg]\dd t,
 \end{split}
\end{equation}
and
\begin{equation}\label{lr:eq:Xi}
 \Xi_d(s)=\frac{\exp\{\Fm(s)\}}{\sqrt{Z_d(s)}}.
\end{equation}
Then $0<\Xi_d(s)<1$ and
\begin{equation}\label{lr:eq:main}
 \sqrt{m!}\,\Lambda_m(s/m,d)\downarrow\Xi_d(s),\qquad
 0\le\log\frac{\sqrt{m!}\Lambda_m(s/m,d)}{\Xi_d(s)}
 \le\frac{K_{s,d}}m
\end{equation}
for a finite constant $K_{s,d}$ independent of $m\ge2$.
The integral in \eqref{lr:eq:Fgamma} has the finite Barnes $G$ evaluation
in Proposition~\ref{lr:prop:Barnes}.
\end{theorem}

\begin{corollary}[A lower bound at every rank]\label{lr:cor:finite}
For every triple in \eqref{quad:eq:region},
\begin{equation}\label{lr:eq:finitebound}
 \Lambda_m(\alpha,d)\ge\frac{\Xi_d(m\alpha)}{\sqrt{m!}}.
\end{equation}
If $2\le k\le m$, then
\begin{equation}\label{lr:eq:chain}
 \Xi_d(s)\le\sqrt{m!}\Lambda_m(s/m,d)
 \le\sqrt{k!}\Lambda_k(s/k,d)<1.
\end{equation}
\end{corollary}

The theorem has a uniform extension that identifies the limiting
memory-boundary coefficient. For fixed $m\ge2$ and $d\le0$, let
\[
 C_{m,d}=\lim_{s\uparrow1}\frac{\Lambda_m(s/m,d)}{\sqrt{1-s}}.
\]
Existence and positivity follow from Theorem~\ref{thm:critical}.
The restriction $d\le0$ allows $s\uparrow1$ within the admissible region
and makes the reference density $(1-d)x^{-d}$ bounded, as used in the
uniform estimates.

\begin{theorem}[The memory boundary and the rank limit]\label{thm:boundary-rank}
Fix $d\le0$. There is a finite constant $\mathcal K_d$, independent
of $m\ge2$ and $0<s<1$, such that
\begin{equation}\label{eq:uniform-rank}
 0\le\log\frac{\sqrt{m!}\Lambda_m(s/m,d)}{\Xi_d(s)}
 \le\frac{\mathcal K_d}{m}.
\end{equation}
Consequently,
\begin{equation}\label{eq:boundary-rank}
 \sqrt{m!}C_{m,d}\downarrow\mathcal C_d,
 \qquad
 0\le\log\frac{\sqrt{m!}C_{m,d}}{\mathcal C_d}
 \le\frac{\mathcal K_d}{m},\qquad
 \mathcal C_d=\frac{e^{\mathcal F_d(1)}}\eta
                  \sqrt{\frac{2\eta-1}{2}}.
\end{equation}
As $m\to\infty$ and $s\uparrow1$, independently,
\begin{equation}\label{eq:joint-boundary}
 \frac{\sqrt{m!}\Lambda_m(s/m,d)}{\sqrt{1-s}}
 =\mathcal C_d\,[1+O_d(m^{-1}+1-s)].
\end{equation}
In particular the two iterated limits agree. At $d=0$,
\[
 \mathcal C_0=2^{3/2}\pi^3G(1/2)^8=\frac{2^{11/6}\pi e}{\mathsf A^{12}},\qquad
 \mathsf A=\exp\{1/12-\zeta'(-1)\}.
\]
\end{theorem}

These statements concern approach to the boundary from $s<1$;
they do not assert the original weighted stochastic LIL at $s=1$.

The contributions are as follows.
\begin{enumerate}
\item The finite-rank constant has an exact R\'enyi variational
representation. Bounds on its global maximizing densities give a
quantitative limit to logarithmic energy minus relative entropy.
The Selberg product then evaluates that limit explicitly
(Sections~\ref{sec:variational}--\ref{sec:large-rank}).
\item An interpolation argument gives density bounds uniform as
$s\uparrow1$ for $d\le0$. It yields
Theorem~\ref{thm:boundary-rank}, including the Glaisher--Kinkelin
coefficient and a joint remainder (Section~\ref{sec:critical}).
\item The chaos large-deviation principle identifies $\Lambda_m$
with a logarithmic right-tail constant for the weighted Hermite
limit (Proposition~\ref{prop:tail}). Thus the high-rank formula
also has a distributional interpretation.
\item Explicit finite-rank bounds, fixed-rank Taylor coefficients
and weight concentration are linked to the limiting function
(Sections~\ref{sec:global}--\ref{sec:concentration}). In particular,
the leading rank coefficients of the independent quadratic and cubic
calculations agree with the Taylor coefficients of $\log\Xi_d$.
Detailed special-weight integrations and the full global comparison
calculations are given in the accompanying \emph{Supplementary calculations}.
\end{enumerate}

\subsection{Relation to earlier work}
Mori--Oodaira describe iterated-logarithm endpoints of multiple Wiener
integrals through variational extrema \cite{MO86,MO87}.
Hermite processes arise in noncentral limit theorems for long-memory
Gaussian sequences \cite{Taqqu79}; deterministic integrals with
respect to them are developed in \cite{MaejimaTudor07}, and
\cite{BaiTaqqu14} gives a broader class of Hermite-type processes.
The weighted stochastic input, the rank-two spectral analysis and the
learning applications are given in
\cite{MoldavskayaLIL,MoldavskayaRankTwo,MoldavskayaML}.
The tail identification uses the classical Gaussian-chaos
large-deviation theory of Borell, Ledoux and
Mayer-Wolf--Nualart--P\'erez-Abreu
\cite{Borell78,Ledoux90,MWNP92}.

The attractive logarithmic mean-field argument has predecessors in
the work of Caglioti, Lions, Marchioro and Pulvirenti \cite{CLMP92} and
Kiessling \cite{Kiessling93}. The theorem used here is
Berman's formulation for singular interactions \cite{Berman}.
The Selberg evaluation itself is classical \cite{ForresterWarnaar};
attractive Jacobi integrals and related Barnes formulas already occur
in \cite{FyodorovLeDoussal,RemyZhu,Ostrovsky}.
High-temperature matrix ensembles provide a related mean-field
regime \cite{ABG12,TrinhTrinh21}. For the Jacobi ensembles in
\cite{TrinhTrinh21}, $\beta_NN\to2c$ with $c\ge0$;
our partition function has $\beta_N=-s/N$ and hence $c=-s/2<0$.
The attractive limit follows from Berman's theorem and the Selberg
product. The contribution here is the passage from
weighted iterated-logarithm constants to the explicitly evaluated
free energy, with finite-rank and boundary error control.

\section{Model and deterministic normalization}\label{sec:model}

The probabilistic setting is the following. The sequence $(X_t)$ is
centered, stationary and Gaussian, with $\E X_t^2=1$. Its covariance and
spectral density satisfy
\[
 \rho(h)\sim c h^{-\alpha}L_0(h),\qquad
 f_X(\lambda)\sim c_f|\lambda|^{\alpha-1}L_0(1/|\lambda|),
\]
where $c,c_f>0$ and $L_0$ is eventually positive and slowly varying.
The spectral convention is
$\rho(h)=\int_{-\pi}^{\pi}e^{ih\lambda}f_X(\lambda)\dd\lambda$.
The density is bounded on compact subsets away from zero and belongs
to $W^{1,1}([-\pi,-\delta]\cup[\delta,\pi])$ for every $0<\delta<\pi$.
The centered function $\mathcal G\in L^2(\mathbb R,\phi)$ has expansion
\[
 \mathcal G(x)=\sum_{q=m}^{\infty}\frac{J_q}{q!}H_q(x),\qquad J_m\ne0,
\]
where $\phi$ is standard Gaussian measure. Write
\[
 S_n=\sum_{t=1}^n a_t\mathcal G(X_t),\qquad
 a_t=t^{-d}\ell_a(t),\qquad A_n^2=\Var(S_n),
\]
where $\ell_a$ is eventually positive, slowly varying and eventually
continuously differentiable, with $x\ell_a'(x)/\ell_a(x)\to0$.
We work in the nonlinear parameter region
\begin{equation}\label{quad:eq:region}
 m\geq2,\qquad 0<\alpha<1/m,\qquad 2d+m\alpha<1.
\end{equation}
The convention for the sharp constant is
\begin{equation}\label{quad:eq:lil}
 \limsup_{n\to\infty}
 \frac{|S_n|}{A_n(2\log\log n)^{m/2}}
 =\Lambda_m(\alpha,d)\quad\text{almost surely}.
\end{equation}
These are assumptions (A1)--(A6) of \cite{MoldavskayaLIL}. Both the
covariance and the spectral conditions are imposed. The LIL
\eqref{quad:eq:lil} follows from the nonlinear functional theorem
and scalar corollary of \cite{MoldavskayaLIL}.

Constant positive weights are included in $d=0$, with $\ell_a$ constant.
More generally, $d=0$ also permits nonconstant slowly varying $\ell_a$, subject
to the original hypotheses. Since
\[
 \frac{a_{\lfloor nx\rfloor}}{a_n}\longrightarrow x^{-d},\qquad x>0,
\]
the limiting weight profile is $1$ when $d=0$. The slowly varying factor is
absorbed into $A_n$; hence the same $\Lambda_m(\alpha,0)$ applies.

Write $w(x)=x^{-d}$ and
\begin{align}
 I_{m,\alpha,d}
 &=\int_0^1\int_0^1 w(x)w(y)|x-y|^{-m\alpha}\dd x\dd y\notag\\
 &=\frac{2B(1-d,1-m\alpha)}{2-2d-m\alpha}.\label{quad:eq:I}
\end{align}
The last equality follows by symmetry and $y=xt$ on $0<y<x<1$.
Set
\begin{equation}\label{quad:eq:mu}
 W=\int_0^1w(x)\dd x=\frac1{1-d},\qquad
 \mu(\dd x)=\mu_d(\dd x)=\frac{w(x)}W\dd x,
 \qquad p=\frac m{m-1}.
\end{equation}
Thus $\mu$ is a probability measure and $p$ and $m$ are conjugate exponents.
We also use $\eta=1-d>1/2$ and $s=m\alpha$; thus
$\mu_d(\dd x)=\eta x^{\eta-1}\dd x$ and the region is
$0<s<\min(1,2\eta-1)$. In particular $d$ may be negative without
a lower bound. The notation $s(x)$ used later for a logarithmic row
moment always includes an argument, to distinguish it from $s=m\alpha$.

For completeness the deterministic tensor definition is
\[
 \Lambda_m(\alpha,d)=\sup_{\|\xi\|_{H_{\R}}\le1}
       |\langle Q_1,\xi^{\otimes m}\rangle|,
 \qquad m!\|Q_1\|_2^2=1,
\]
where $H_{\R}=\{\xi\in L^2(\mathbb R;\mathbb C):
\xi(-u)=\overline{\xi(u)}\}$ and, up to an irrelevant overall sign,
\[
 Q_1(u_1,\ldots,u_m)=
 \frac{\displaystyle\int_0^1x^{-d}e^{ix(u_1+\cdots+u_m)}\dd x}
      {\sqrt{m!c_\alpha^m I_{m,\alpha,d}}}
 \prod_{j=1}^m|u_j|^{(\alpha-1)/2},
 \quad c_\alpha=2\Gamma(\alpha)\cos(\pi\alpha/2).
\]
The integral kernel belongs to $L^2$ in the stated region, by the
Riesz energy identity and \eqref{quad:eq:I}. This definition makes all
the deterministic theorems below meaningful independently of the
probabilistic input. At rank one the same normalization gives
$\Lambda_1=1$ exactly.

We use both $Z_m(\alpha,d)$ and $Z_d(s)$ below; they denote the
same quantity when $s=m\alpha$. The letter $\eta$ always denotes
$1-d$. The moment symbols $a,b,c$ introduced in
Section~\ref{sec:local} are local logarithmic moments.

\section{An exact variational reformulation}\label{sec:variational}

Let
\[
 (K_\alpha h)(x)=\int_0^1|x-y|^{-\alpha}h(y)\dd y.
\]
The tensor norm in \cite{MoldavskayaLIL} has the interval representation
\begin{equation}\label{quad:eq:interval}
 \Lambda_m(\alpha,d)=\frac1{\sqrt{m!I_{m,\alpha,d}}}
 \sup_{h\in L^2(0,1),\,h\ne0}
 \frac{\left|\int_0^1 w(x)(K_\alpha h(x))^m\dd x\right|}
 {\ip{h}{K_\alpha h}_{L^2}^{m/2}}.
\end{equation}
To check the normalization, put $c_\alpha=2\Gamma(\alpha)\cos(\pi\alpha/2)$.
Equip the spectral space $H_{\R}$ with its real scalar product
\[
 \langle\xi,\zeta\rangle_{H_{\R}}
 =\operatorname{Re}\int_{\R}\xi(u)\overline{\zeta(u)}\dd u.
\]
The spectral trial functions
\[
 \xi_h(u)=|u|^{(\alpha-1)/2}\int_0^1h(y)e^{iyu}\dd y
\]
satisfy $\norm{\xi_h}^2=c_\alpha\ip{h}{K_\alpha h}$. In the contraction
with the variance-normalized kernel $Q_1$, each coordinate contributes
$c_\alpha K_\alpha h$, whereas $Q_1$ contains the factor
$c_\alpha^{-m/2}/\sqrt{m!I_{m,\alpha,d}}$. Division by
$\norm{\xi_h}^m$ cancels all powers of $c_\alpha$, giving
\eqref{quad:eq:interval}; the sign of the Hermite coefficient disappears under
the absolute value. These identities may first be regularized in frequency
and then extended by the Riesz energy identity. The closed span of these
trial functions is the relevant spectral subspace: an orthogonal component
has zero interval transform, so its tensor contraction vanishes.

Define the real integral operator and the scalar normalization
\begin{align}
 (R_\alpha f)(x)&=\int_0^1|x-y|^{-\alpha}f(y)\mu(\dd y),\label{quad:eq:R}\\
 Z_m(\alpha,d)&=\iint|x-y|^{-m\alpha}\mu(\dd x)\mu(\dd y)
                 =\frac{I_{m,\alpha,d}}{W^2}.\label{quad:eq:Z}
\end{align}
All spaces and pairings below are real. Set
\begin{equation}\label{quad:eq:D}
 D_m(\alpha,d)=\max_{\norm{f}_{L^p(\mu)}=1}
 B_\alpha(f,f),\qquad
 B_\alpha(f,g)=\int fR_\alpha g\dd\mu.
\end{equation}

\begin{proposition}\label{quad:prop:dual}
The maximum in \eqref{quad:eq:D} exists and has a nonnegative maximizer. Moreover,
\begin{equation}\label{quad:eq:exact}
 \boxed{\displaystyle
 \Lambda_m(\alpha,d)=
 \frac{D_m(\alpha,d)^{m/2}}{\sqrt{m!Z_m(\alpha,d)}}.}
\end{equation}
Every nonnegative normalized maximizer satisfies
\begin{equation}\label{quad:eq:Euler}
 R_\alpha f=D_m(\alpha,d)f^{p-1}.
\end{equation}
\end{proposition}

\begin{proof}
The kernel belongs to $L^m(\mu\otimes\mu)$ by \eqref{quad:eq:region} and
\eqref{quad:eq:I}. Apply H\"older on the product space to the two
factors $|x-y|^{-\alpha}$ and $f\otimes g$, with conjugate exponents $m,p$.
Since $\|f\otimes g\|_{L^p(\mu\otimes\mu)}=\|f\|_p\|g\|_p$, this gives
\[
 |B_\alpha(f,g)|\leq Z_m^{1/m}\norm f_p\norm g_p.
\]
For two kernels $k,\widetilde k\in L^m(\mu\otimes\mu)$, pointwise
H\"older followed by integration in $x$ gives
\[
 \|(R_k-R_{\widetilde k})f\|_m
 \le \|k-\widetilde k\|_{L^m(\mu\otimes\mu)}\|f\|_p.
\]
Finite sums of rectangle indicators are dense in this kernel space
and define finite-rank operators. Hence $R_\alpha:L^p(\mu)\to L^m(\mu)$
is compact. On the weakly compact unit ball of $L^p$,
$B_\alpha(f,f)$ is weakly continuous: if $f_n\rightharpoonup f$, compactness
gives $R_\alpha f_n\to R_\alpha f$ in $L^m$, and
\[
 B_\alpha(f_n,f_n)-B_\alpha(f,f)
 =\int f_n(R_\alpha f_n-R_\alpha f)\dd\mu
   +\int(f_n-f)R_\alpha f\dd\mu\longrightarrow0.
\]
A maximum exists; it is positive,
and homogeneity puts every maximizer on the unit sphere. Replacing $f$ by
$|f|$ cannot decrease the objective. Differentiating the constraint
$\int|f|^p\dd\mu=1$ gives
$2R_\alpha f=\lambda p|f|^{p-2}f$. Pairing with $f$ yields
$\lambda p=2D_m$, hence \eqref{quad:eq:Euler}.

The Riesz kernel is positive definite: for bounded real $f$, this follows
from its nonnegative Fourier density applied to $wf\in L^2(0,1)$
extended by zero, since $d<1/2$;
the assertion extends to $L^p(\mu)$ by density and the preceding bound.
The Cauchy--Schwarz inequality for this positive semidefinite form
makes $B_\alpha$ a scalar product on $L^p(\mu)/\mathcal N$, where
$\mathcal N=\{f:B_\alpha(f,f)=0\}$. Let $\mathcal H$ be its completion
and $J:L^p(\mu)\to\mathcal H$ the canonical map.
Then
\[
 \norm{Jf}_{\mathcal H}^2=B_\alpha(f,f),\quad
 \norm J^2=D_m.
\]
Identify $\mathcal H^*$ with $\mathcal H$ and $(L^p(\mu))^*$ with
$L^m(\mu)$. The Banach adjoint $A=J^*:\mathcal H\to L^m(\mu)$ is
characterized by
\[
 \int f(Au)\dd\mu=\langle Jf,u\rangle_{\mathcal H},
 \qquad f\in L^p(\mu),\quad u\in\mathcal H.
\]
Taking $u=Jg$ shows $AJ=R_\alpha$. Consequently $\norm A^2=D_m$, and
\begin{equation}\label{quad:eq:tensbound}
 \left|\int(Au)^m\dd\mu\right|
 \leq\norm{Au}_m^m\leq D_m^{m/2}\norm u_{\mathcal H}^m.
\end{equation}
For a nonnegative maximizer $f$, take $u=Jf/\sqrt{D_m}$. Its norm is one,
and \eqref{quad:eq:Euler} implies
\[
 Au=\sqrt{D_m}\,f^{p-1}\geq0,\qquad
 \int(Au)^m\dd\mu=D_m^{m/2}\int f^{m(p-1)}\dd\mu=D_m^{m/2}.
\]
Here $m(p-1)=p$. This proves equality in \eqref{quad:eq:tensbound}, including
odd $m$.

Finally, for bounded $g$ put $h=wg$. Then
\[
 K_\alpha h=W R_\alpha g,\qquad
 \ip h{K_\alpha h}=W^2 B_\alpha(g,g).
\]
Substitution into \eqref{quad:eq:interval}, using $I=W^2Z_m$, gives
\[
 \frac1{\sqrt{m!Z_m}}
 \frac{\left|\int(R_\alpha g)^m\dd\mu\right|}{B_\alpha(g,g)^{m/2}}.
\]
Bounded $g$ are dense in $L^p(\mu)$, hence their images under $J$ are
dense in $\mathcal H$. Conversely, for $h\in L^2(0,1)$, truncate $h/w$
to $[-n,n]$ and call the result $g_n$. Since $w>0$ almost everywhere,
$wg_n\to h$ pointwise and $|wg_n|\le|h|$, so convergence holds in $L^2$.
Since $\alpha<1/2$,
$K_\alpha:L^2\to L^\infty$ is bounded; with $w\in L^1$, both the
interval numerator and energy pass to these approximations. Taking
suprema proves \eqref{quad:eq:exact}.

The spectral supremum is attained as well. The map
$Jg\mapsto\xi_{wg}/(W\sqrt{c_\alpha})$, initially defined for
bounded $g$, is an isometry into $H_{\R}$ and extends to $\mathcal H$.
The tensor contraction extends continuously along this isometry.
Thus the unit vector constructed above also gives a spectral unit
vector attaining the positive value $\Lambda_m$.
\end{proof}

\subsection{The energy-space interpretation}
The strictly positive operator $K_\alpha$ on $L^2(0,1)$ defines the
energy space
\[
 \mathcal E_\alpha=\operatorname{Ran}(K_\alpha^{1/2}),\qquad
 \|\psi\|_{\mathcal E_\alpha}=
          \|K_\alpha^{-1/2}\psi\|_2.
\]
It is complete in this norm, although its range need not be closed in
$L^2$. The functions $\psi=K_\alpha h$ are dense in this energy norm.
The Hilbert-space construction in Proposition~\ref{quad:prop:dual}
therefore extends \eqref{quad:eq:interval} to
\[
 \Lambda_m(\alpha,d)=\frac1{\sqrt{m!I_{m,\alpha,d}}}
 \sup_{\|\psi\|_{\mathcal E_\alpha}\le1}
       \left|\int_0^1w(x)\psi(x)^m\dd x\right|.
\]
In particular, it is the $m$th power of the sharp norm of the embedding
$\mathcal E_\alpha\hookrightarrow L^m(w\dd x)$, divided by
$\sqrt{m!I_{m,\alpha,d}}$; the equality for odd $m$ follows from the
positive maximizer in that proposition. For $\psi$ in the operator
domain of $K_\alpha^{-1}$ the squared energy norm is
$\langle K_\alpha^{-1}\psi,\psi\rangle$, with the quadratic-form
interpretation on the full energy space. This is an inverse-potential
energy on an interval.

\subsection{Transfer between Hermite ranks}
\begin{proposition}\label{prop:rank-transfer}
For integers $2\le k\le m$ and admissible $(m,\alpha,d)$,
\[
 \sqrt{m!}\Lambda_m(\alpha,d)
 \le\sqrt{k!}\Lambda_k(m\alpha/k,d).
\]
\end{proposition}
\begin{proof}
There is equality for $k=m$. Otherwise put $r=m/k>1$,
$p_j=j/(j-1)$, and $u=f^{p_m/p_k}$ for a nonnegative $f$ with
$\|f\|_{p_m}=1$. Then $\|u\|_{p_k}=1$. The integrand of
$B_\alpha(f,f)$ factors as
\[
 \bigl[|x-y|^{-r\alpha}u(x)u(y)\bigr]^{1/r}
 \bigl[f(x)^{p_m}f(y)^{p_m}\bigr]^{1-1/r}.
\]
H\"older gives $D_m(\alpha,d)\le D_k(r\alpha,d)^{1/r}$.
The two $Z$ factors agree because $k(r\alpha)=m\alpha$.
Proposition~\ref{quad:prop:dual} proves the result.
\end{proof}
In particular any rank-two upper estimate at $\beta=m\alpha/2$
transfers to every higher rank through this comparison inequality.

\subsection{Logarithmic tails of the weighted Hermite limit}
Let $Y_{m,\alpha,d}=I_m(Q_1)$, with $Q_1$ as in
Section~\ref{sec:model} and with its displayed positive normalization.
Thus $\E Y_{m,\alpha,d}^2=m!\|Q_1\|_2^2=1$.
The underlying real isonormal Gaussian process $\mathbb W$ has
Cameron--Martin space identified with $H_{\R}$, equipped with the
real scalar product above. We use the Wiener--It\^o normalization
\[
 I_m(h^{\otimes m})=H_m(\mathbb W(h)),\qquad \|h\|_{H_{\R}}=1,
\]
where $H_m$ is the monic probabilists' Hermite polynomial, defined by
\[
 e^{tx-t^2/2}=\sum_{j=0}^{\infty}H_j(x)\frac{t^j}{j!}.
\]
Consequently $\E[I_m(f)I_m(g)]=m!\langle f,g\rangle$ for symmetric
kernels, and the homogeneous Cameron--Martin skeleton is
$h\mapsto\langle Q_1,h^{\otimes m}\rangle$.
The sign fixes the orientation of the right tail; changing the sign
of the leading Hermite coefficient reflects the limiting law.

\begin{proposition}[Tail interpretation]\label{prop:tail}
For every admissible $(m,\alpha,d)$,
\begin{equation}\label{eq:tail}
 \lim_{t\to\infty}t^{-2/m}\log\Pr\{Y_{m,\alpha,d}>t\}
 =-\frac12\Lambda_m(\alpha,d)^{-2/m}.
\end{equation}
The same limit holds with $|Y_{m,\alpha,d}|>t$.
\end{proposition}
\begin{proof}
The Gaussian-chaos large-deviation principle
\cite[unnumbered theorem, p.~4]{Ledoux90} gives the rate function
\[
 I(y)=\frac12\inf\{\|h\|_{H_\R}^2:
                  \langle Q_1,h^{\otimes m}\rangle=y\}
\]
for $\varepsilon^{m/2}I_m(Q_1)$ on $\mathbb R$ with its usual topology,
at speed $\varepsilon^{-1}$. In particular, its upper bound applies
to every closed set, including the unbounded tail sets below.
This is the scalar form of the process results in \cite{MWNP92}.
Proposition~\ref{quad:prop:dual} constructs a unit vector attaining
the positive value $\Lambda_m$; its absolute value is the full
supremum. The skeleton $P(h)=\langle Q_1,h^{\otimes m}\rangle$ is
continuous on $H_{\R}$: telescoping the tensor powers gives
\[
 |P(h)-P(g)|\le m\|Q_1\|\max(\|h\|,\|g\|)^{m-1}\|h-g\|.
\]
For $y>0$, its bound
$|\langle Q_1,h^{\otimes m}\rangle|\le\Lambda_m\|h\|^m$ and a
rescaling of that unit maximizer give
\[
 I(y)=\frac12\left(\frac{y}{\Lambda_m}\right)^{2/m}.
\]
In particular,
\[
 \inf_{y>1}I(y)=\inf_{y\ge1}I(y)
 =\frac12\Lambda_m^{-2/m}.
\]
Apply the lower and upper large-deviation bounds to $(1,\infty)$
and $[1,\infty)$, with $\varepsilon=t^{-2/m}$.
For the absolute tail, the bound by $\Lambda_m\|h\|^m$ and the
positive maximizer likewise give equal infima over $|y|>1$ and
$|y|\ge1$. Applying the same bounds proves the assertion.
\end{proof}

At odd rank the law is symmetric, since $\mathbb W\mapsto-\mathbb W$
preserves its Gaussian law and changes the sign of $I_m(Q_1)$.
At even rank the chosen orientation and the positive maximizer give
the stated right-tail constant.

At $d=0$, $Y_{m,\alpha,0}$ is the variance-one time-one marginal
of the Hermite process with Hurst index $H=1-m\alpha/2$.
For general $d$, it is the standardized deterministic weighted
integral over $[0,1]$ of that process; its energy integral is
$I_{m,\alpha,d}$. This follows directly from the time integral
in $Q_1$ and the deterministic-integral construction in
\cite{MaejimaTudor07}. At rank two, the rate in \eqref{eq:tail}
is $1/(2\Lambda_2)$, agreeing with the exponential-moment
threshold in \cite{MoldavskayaRankTwo}; see also \cite{VT13}.

Writing $\tau_m(s,d)=\tfrac12\Lambda_m(s/m,d)^{-2/m}$,
Theorem~\ref{lr:thm:main} gives the further interpretation
\[
 \log\tau_m(s,d)=\frac{\log(m!)}m-\log2
                 -\frac{2\log\Xi_d(s)}m+O_{s,d}(m^{-2}).
\]
Here the threshold $t\to\infty$ is taken at each fixed $m$.
No joint limit of $t$ and $m$ is asserted.

\section{High-rank limit and its explicit evaluation}\label{sec:large-rank}
\subsection{R\'enyi representation and uniform density bounds}\label{lr:sec:entropy}

Write $p_m=m/(m-1)$ when comparing ranks, and put $\ell(x,y)=-\log|x-y|$.
For a probability density $\rho$ with respect to $\mu_d$, write
\[
 H(\rho)=\int\rho\log\rho\dd\mu_d,
 \qquad D_p(\rho\|\mu_d)=\frac1{p-1}\log\int\rho^p\dd\mu_d,
\]
with $0\log0=0$. Thus $H$ is relative entropy and $D_p$ is
R\'enyi divergence.

\begin{proposition}[Exact R\'enyi representation]\label{lr:prop:renyi}
For the parameters in \eqref{quad:eq:region},
\begin{equation}\label{lr:eq:renyi}
 \frac m2\log D_m(s/m,d)=\sup_\rho J_m(\rho),
\end{equation}
where the supremum is over probability densities in $L^{p_m}(\mu_d)$ and
\begin{equation}\label{lr:eq:Jm}
 J_m(\rho)=\frac m2\log\iint e^{s \ell(x,y)/m}
          \rho(x)\rho(y)\mu_d(\dd x)\mu_d(\dd y)
          -D_{p_m}(\rho\|\mu_d).
\end{equation}
\end{proposition}

\begin{proof}
The quotient $B_\alpha(f,f)/\|f\|_{p_m}^2$ is
homogeneous. Every nonzero nonnegative $f$ can be divided by its
$L^1$ norm to obtain a probability density. Multiplying the logarithm
of the quotient by $m/2$ leaves the norm contribution
\[
 m\log\|\rho\|_{p_m}
 =(m-1)\log\int\rho^{p_m}\dd\mu_d=D_{p_m}(\rho\|\mu_d).
\]
This proves the representation.
\end{proof}

\begin{lemma}[Uniform bounds for maximizing densities]\label{lr:lem:density}
Define
\begin{equation}\label{lr:eq:C}
 C_s=\sup_{x\in[0,1]}\int_0^1|x-y|^{-s}\mu_d(\dd y),
 \qquad M=C_s^2.
\end{equation}
Then $1\le C_s<\infty$. Every maximizing probability density
$\rho_m$ in \eqref{lr:eq:renyi} satisfies
\begin{equation}\label{lr:eq:bounddensity}
 D_{p_m}(\rho_m\|\mu_d)\le\log C_s,
 \qquad \|\rho_m\|_\infty\le C_s^{2(m-1)/m}\le M.
\end{equation}
\end{lemma}

\begin{proof}
If $d\le0$, the density of $\mu_d$ is bounded, and $s<1$ suffices
for the row bound. If $d>0$, splitting the $y$ integral into
$y<x/2$, $x/2\le y\le2x$, and $y>2x$ shows it is bounded
provided $s+d<1$. This follows from $s+2d<1$.
At $x=0$ the integral is $\eta/(\eta-s)$; the same condition controls
the limit as $x\downarrow0$. The kernel is at least one, so $C_s\ge1$.

By Proposition~\ref{quad:prop:dual}, a maximizing density exists.
Its Euler equation, after $L^1$ normalization, is
\begin{equation}\label{lr:eq:eulerrho}
 \rho_m(x)=\frac{(R_\alpha\rho_m(x))^{m-1}}{N_m},
 \qquad N_m=\int(R_\alpha\rho_m)^{m-1}\dd\mu_d\ge1.
\end{equation}
The final inequality uses $R_\alpha\rho_m\ge1$.
H\"older with exponents $m,p_m$ gives
\[
 R_\alpha\rho_m(x)\le C_s^{1/m}\|\rho_m\|_{p_m}
 =\exp\left\{\frac{\log C_s+D_{p_m}(\rho_m\|\mu_d)}m\right\}.
\]
Since the quadratic energy is at most this row supremum,
\[
 J_m(\rho_m)\le\tfrac12\log C_s
                 -\tfrac12D_{p_m}(\rho_m\|\mu_d).
\]
The trial $\rho=1$ gives $J_m(\rho_m)\ge J_m(1)\ge0$.
This proves the entropy estimate; \eqref{lr:eq:eulerrho} then gives
the asserted uniform density estimate.
\end{proof}

\begin{lemma}[Uniform comparison of the two functionals]\label{lr:lem:error}
For probability densities $0\le\rho\le M$, define
\[
 J_\infty(\rho)=\frac s2\iint \ell(x,y)\rho(x)\rho(y)
                         \mu_d(\dd x)\mu_d(\dd y)-H(\rho).
\]
Set
\[
 A_{s,d,M}=\frac{s^2M^2}{4}
       \iint \ell^2e^{s\ell/2}\dd\mu_d^{\otimes2},
 \qquad B_M=\frac4{e^2}+M^2(\log M)^2.
\]
These constants are finite, and for every $m\ge2$,
\begin{equation}\label{lr:eq:error}
 J_\infty(\rho)-\frac{B_M}{m}
 \le J_m(\rho)\le J_\infty(\rho)+\frac{A_{s,d,M}}m.
\end{equation}
\end{lemma}

\begin{proof}
The logarithmic moment in $A_{s,d,M}$ is finite: it is the second
derivative of $Z_d(t)$ at $t=s/2$, a point strictly inside its
integrability range. Let expectations below be under the product
probability measure with density $\rho(x)\rho(y)$.
Jensen's inequality and the elementary exponential remainder imply
\[
 0\le\frac m2\log\E e^{s\ell/m}-\frac s2\E\ell
 \le\frac{s^2}{4m}\E[\ell^2e^{s\ell/m}]
 \le\frac{A_{s,d,M}}m.
\]
For the upper estimate, use also $D_{p_m}\ge H$.

For the lower estimate put $h(t)=\log\int\rho^{1+t}\dd\mu_d$,
$0\le t\le1$. Then $h(0)=0$, $h'(0)=H(\rho)$, and $h''(t)$
is the variance of $\log\rho$ under density proportional to
$\rho^{1+t}$. Since $\int\rho^{1+t}\dd\mu_d\ge1$,
\[
 0\le h''(t)\le\int\rho^{1+t}(\log\rho)^2\dd\mu_d\le B_M.
\]
The last bound uses $u(\log u)^2\le4/e^2$ for $0\le u\le1$.
Taylor's formula at $t=1/(m-1)$ yields
\[
 0\le D_{p_m}(\rho\|\mu_d)-H(\rho)
 \le\frac{B_M}{2(m-1)}\le\frac{B_M}{m}.
\]
The argument allows zeros of $\rho$, by continuous extension or
truncation. Combining the estimates proves \eqref{lr:eq:error}.
\end{proof}

\subsection{The entropy limit and the convergence rate}\label{lr:sec:limit}

Define the variational value
\begin{equation}\label{lr:eq:Fvar}
 F^{\mathrm{var}}_d(s)=\sup_{\rho:\,H(\rho)<\infty}J_\infty(\rho),
\end{equation}
where $\rho$ is a probability density with respect to $\mu_d$.
The entropy inequality on the product probability space gives
\begin{equation}\label{lr:eq:entenergy}
 t\iint\ell\rho\otimes\rho\dd\mu_d^{\otimes2}
 \le2H(\rho)+\log Z_d(t),\qquad 0<t<1.
\end{equation}
Thus the energy is finite whenever $H(\rho)$ is finite. Choosing
$s<t<1$ also proves $F^{\mathrm{var}}_d(s)<\infty$.
More precisely, for such $t$,
\[
 -J_\infty(\rho)
 \ge \left(1-\frac{s}{t}\right)H(\rho)
       -\frac{s}{2t}\log Z_d(t).
\]
Thus bounded free energy implies bounded entropy. In the free-energy
formulation below, the functional equals $-J_\infty(\rho)$ on
$\nu=\rho\mu_d$ of finite relative entropy and is assigned $+\infty$
otherwise. This extension is consistent with weak limits: if
finite-entropy measures converge weakly to a measure of infinite
relative entropy, lower semicontinuity of entropy and the displayed
bound force their free energies to tend to $+\infty$.

\begin{proposition}[Attainment and quantitative convergence]\label{lr:prop:limit}
The supremum in \eqref{lr:eq:Fvar} is attained by a density
$\rho_*$ with $0\le\rho_*\le M=C_s^2$. Moreover,
\begin{equation}\label{lr:eq:ratefree}
 F^{\mathrm{var}}_d(s)-\frac{B_M}{m}
 \le\frac m2\log D_m(s/m,d)
 \le F^{\mathrm{var}}_d(s)+\frac{A_{s,d,M}}m.
\end{equation}
Every weak-* limit point in $L^\infty(\mu_d)$ of maximizing
densities $\rho_m$ maximizes \eqref{lr:eq:Fvar}.
\end{proposition}

\begin{proof}
First, bounded probability densities suffice in \eqref{lr:eq:Fvar}.
For any density of finite entropy, normalize $\rho\wedge T$ to
have integral one and let $T\to\infty$. The entropy converges
by truncation, and the energy converges by monotone convergence
before normalization, using its finiteness from \eqref{lr:eq:entenergy}.

Lemma~\ref{lr:lem:density} gives a weak-* convergent subsequence
$\rho_{m_j}\to\rho_*$ in $L^\infty(\mu_d)$, with
$0\le\rho_*\le M$ and integral one. On this bounded set the
logarithmic energy is continuous for this convergence. To verify
this, define $\ell_T=\min(\ell,T)$ off the diagonal and $\ell_T(x,x)=T$;
this is a continuous function on $[0,1]^2$. Weak-* convergence of the
densities gives weak convergence of $\rho_{m_j}\mu_d$ by testing
against continuous functions. Their product measures converge weakly
as well, by uniform approximation with finite sums of product
functions. Hence the truncated energies converge. The omitted
tail is at most
\[
 M^2\iint(\ell-T)_+\dd\mu_d^{\otimes2}\longrightarrow0
\]
uniformly in the sequence. Relative entropy is weakly lower
semicontinuous. Lemma~\ref{lr:lem:error} therefore gives
\[
 \limsup_{j\to\infty}J_{m_j}(\rho_{m_j})\le J_\infty(\rho_*).
\]
For every fixed bounded probability density $\rho$, the same lemma,
with its own finite bound in place of $M$, gives $J_m(\rho)\to
J_\infty(\rho)$. Maximality of $\rho_m$ and the bounded-density
approximation just proved show
\[
 \liminf_{m\to\infty}\sup J_m\ge F^{\mathrm{var}}_d(s).
\]
Consequently $\rho_*$ attains \eqref{lr:eq:Fvar}, and the supremal
values converge. Finally apply Lemma~\ref{lr:lem:error} to $\rho_m$
for the upper bound, and to $\rho_*$ as a trial for the lower
bound. This proves \eqref{lr:eq:ratefree}.
\end{proof}

The compactness argument establishes convergence of the global maxima
and identifies every limit point of maximizing densities.

\subsection{Selberg evaluation of the limiting free energy}\label{lr:sec:Selberg}

For $N\ge2$ consider the partition function
\begin{equation}\label{lr:eq:partition}
 \mathcal Z_N(s,d)=\int_{[0,1]^N}
 \prod_{i<j}|x_i-x_j|^{-s/N}\prod_{i=1}^N\mu_d(\dd x_i).
\end{equation}
The interaction is attractive: configurations with nearby coordinates
receive larger weight. The following use of a known mean-field
theorem is the bridge between the variational limit and this exactly
integrable expression.

\begin{lemma}[Mean-field identification]\label{lr:lem:meanfield}
For fixed admissible $(s,d)$,
\begin{equation}\label{lr:eq:freeenergy}
 F^{\mathrm{var}}_d(s)=\lim_{N\to\infty}\frac1N\log\mathcal Z_N(s,d).
\end{equation}
\end{lemma}

\begin{proof}
The lower bound is elementary. For any bounded probability density
$\rho$, the Gibbs variational inequality applied to
$(\rho\mu_d)^{\otimes N}$ gives
\[
 \frac1N\log\mathcal Z_N(s,d)
 \ge\frac{s(N-1)}{2N}\iint\ell\rho\otimes\rho\dd\mu_d^{\otimes2}
       -H(\rho).
\]
Let $N\to\infty$ and take the supremum; the bounded-density
approximation in Proposition~\ref{lr:prop:limit} gives
$\liminf N^{-1}\log\mathcal Z_N\ge F_d^{\mathrm{var}}(s)$.
For the reverse inequality, apply Berman~\cite[Theorem 1.5 and its proof in Section 2.7]{Berman}
on the compact metric space $[0,1]$, with reference measure $\mu_d$,
interaction $W(x,y)= \ell(x,y)$, and inverse temperature $\beta=-s/2$.
His ordered-pair Hamiltonian is
$H^{(N)}=N^{-1}\sum_{i\ne j}W(x_i,x_j)$; hence
$e^{-\beta H^{(N)}}$ is exactly the integrand in \eqref{lr:eq:partition}.
Choose $q$ with $s/2<q<\min(1,\eta)$, possible because
$s/2<1/2<\eta$, and put $\beta_0=-q$; thus
$\beta>\beta_0$. Then
\[
 \sup_x\int e^{-\beta_0 W(x,y)}\mu_d(\dd y)
 =\sup_x\int|x-y|^{-q}\mu_d(\dd y)<\infty.
\]
For $\eta\ge1$ this follows from boundedness of the reference density
and $q<1$. For $\eta<1$, splitting at $x/2$ and $2x$ gives a uniform
bound because $q<\eta$ and $q<1$; at $x=0$ the integral is
$\eta/(\eta-q)$. Since $0\le W\le q^{-1}e^{qW}$, this also implies
$W\in L^1(\mu_d^{\otimes2})$.
The interaction is symmetric and lower semicontinuous, with value
$+\infty$ on the diagonal, which has zero product measure. The energy
per particle satisfies
\[
 \int\frac{H^{(N)}}N\dd\nu^{\otimes N}
 =\frac{N-1}N\iint W\dd\nu^{\otimes2}
\]
whenever the right-hand integral is finite. This verifies the
mean-energy assumption (H1), with $E(\nu)=\iint W\dd\nu^{\otimes2}$
and $E(\mu_d)<\infty$. The specialized pair-interaction theorem applies
to a symmetric lower semicontinuous interaction on a compact metric
space with the displayed uniform exponential bound and
$\beta>\beta_0$; its reference probability measure is unrestricted.
Here $\mu_0=\mu_d$ and the inverse temperatures are constant in $N$.
The unnormalized free-energy convergence
proved there, with the finite-entropy domain and extension specified
above, gives
\[
 -\lim_{N\to\infty}\frac1N\log\mathcal Z_N(s,d)
 =\inf_{\rho:\,H(\rho)<\infty}
 \left\{H(\rho)-\frac{s}{2}
        \iint\ell\rho\otimes\rho\dd\mu_d^{\otimes2}\right\}
 =-F_d^{\mathrm{var}}(s).
\]
This proves \eqref{lr:eq:freeenergy}.
\end{proof}

The scaling $\beta=-s/2$ places the partition function in the
attractive regime. The Gibbs inequality supplies the lower bound on
the limiting log-partition function; Berman supplies the reverse bound.
Their equality, followed by rank monotonicity, yields the finite-rank
lower estimate in Corollary~\ref{lr:cor:finite}.

\begin{proposition}[Gamma and Barnes evaluations]\label{lr:prop:Barnes}
The variational value $F^{\mathrm{var}}_d(s)$ equals $\Fm(s)$ in
\eqref{lr:eq:Fgamma}. Let $\omega=s/2$, and for $z>0$ put
\[
 \mathfrak g(z)=(z-1)\log\Gamma(z)-\log G(z),\qquad\varphi(z)=z\log z,
\]
where $G$ is the Barnes $G$ function, normalized by
$G(1)=1$ and $G(z+1)=\Gamma(z)G(z)$. Then
\begin{equation}\label{lr:eq:Barnes}
 \begin{split}
 \Fm(s)={}&\log(2\pi \eta)-\frac2\omega \mathfrak g(1-\omega)\\
 &+\frac1\omega\bigl[\mathfrak g(\eta)-2\mathfrak g(\eta-\omega)+\mathfrak g(\eta-2\omega)
                    +\varphi(\eta-2\omega)-\varphi(\eta-\omega)\bigr].
 \end{split}
\end{equation}
All arguments in this formula are positive in \eqref{quad:eq:region}.
\end{proposition}

\begin{proof}
By the Selberg integral evaluation \cite{ForresterWarnaar},
\begin{equation}\label{lr:eq:Selbergproduct}
 \mathcal Z_N(s,d)=\eta^N\prod_{j=0}^{N-1}
 \frac{\Gamma(\eta-j\omega/N)\Gamma(1-j\omega/N)
       \Gamma(1-(j+1)\omega/N)}
 {\Gamma(\eta+1-(N+j-1)\omega/N)\Gamma(1-\omega/N)}.
\end{equation}
In the standard Selberg product, the numerator factors are
$\Gamma(a+jc)\Gamma(b+jc)\Gamma(1+(j+1)c)$ and the denominator is
$\Gamma(a+b+(N+j-1)c)\Gamma(1+c)$. Substituting
$(a,b,c)=(\eta,1,-\omega/N)$ gives \eqref{lr:eq:Selbergproduct}.
Its integrability
conditions hold: $\omega<1$ and $\omega<\eta$, so
$-\omega/N> -\min(1/N,\eta/(N-1),1/(N-1))$.
After taking logarithms and dividing by $N$, the sums in
\eqref{lr:eq:Selbergproduct} are Riemann sums of smooth functions
on $[0,1]$. Their gamma arguments are bounded away from zero
for fixed admissible $(s,d)$. The limit is \eqref{lr:eq:Fgamma}.
Lemma~\ref{lr:lem:meanfield} identifies it with the variational value.

For the second formula, the standard Barnes integral identity
\cite[Section 5.17]{DLMF} gives
\[
 \mathfrak g'(z)=\log\Gamma(z)+z-\frac12-\frac12\log(2\pi).
\]
Thus a primitive of $\log\Gamma(z)$ is
\[
 P(z)=\mathfrak g(z)-\frac{z^2}{2}+\frac z2+\frac z2\log(2\pi).
\]
The three integrals in \eqref{lr:eq:Fgamma} become
\[
 \frac{P(\eta)-P(\eta-\omega)}\omega,
 \quad\frac{2(P(1)-P(1-\omega))}\omega,
 \quad-\frac{P(\eta+1-\omega)-P(\eta+1-2\omega)}\omega.
\]
Use $\mathfrak g(1)=0$ and
$\mathfrak g(z+1)-\mathfrak g(z)=z\log z=\varphi(z)$, and collect the elementary
polynomial terms. This yields \eqref{lr:eq:Barnes}.
\end{proof}

The gamma integral also gives the continuous extension
$\mathcal F_d(0)=0$ and $Z_d(0)=1$, hence $\Xi_d(0)=1$.
Differentiation at zero gives
\[
 \mathcal F_d'(0)=\frac12\left.\frac{\dd}{\dd s}\log Z_d(s)\right|_{s=0},
 \qquad \left.\frac{\dd}{\dd s}\log\Xi_d(s)\right|_{s=0}=0.
\]
The first nonzero Taylor terms are evaluated in
Section~\ref{sec:local}.

\begin{proof}[Proof of Theorem~\ref{lr:thm:main} and Corollary~\ref{lr:cor:finite}]
Combining \eqref{quad:eq:exact} and Proposition~\ref{lr:prop:limit}
gives convergence to $\exp\{F_d^{\mathrm{var}}(s)\}/\sqrt{Z_d(s)}$.
Proposition~\ref{lr:prop:Barnes} identifies this limit with $\Xi_d(s)$.
The upper estimate in \eqref{lr:eq:ratefree} gives the logarithmic
error bound with $K_{s,d}=A_{s,d,C_s^2}$. Finally,
Proposition~\ref{prop:rank-transfer} makes the sequence nonincreasing,
so convergence is from above and the logarithmic error is nonnegative.
The explicit expression is positive. Taking the limit in the rank
comparison with rank two and using Theorem~\ref{thm:finite-bounds}
gives $\Xi_d(s)<1$. The same rank comparison yields
\eqref{lr:eq:finitebound} and \eqref{lr:eq:chain}.
\end{proof}

\begin{remark}[Relation to interval chaos moments]
R\'emy and Zhu~\cite[Eq.~(1.8)]{RemyZhu} denote the weighted
moments of interval Gaussian multiplicative chaos by $M$.
In their notation,
\[
 \mathcal Z_N(s,d)=\eta^N
 M\left(\sqrt{2s/N},N,-d,0\right).
\]
This identifies the partition function with their weighted moment
integral. Its connection to the high-rank constant follows through
the passage from the R\'enyi optimization \eqref{lr:eq:renyi}
to the entropy variational problem \eqref{lr:eq:Fvar}.
\end{remark}

\section{Behavior at the memory boundary}\label{sec:critical}
The explicit lower bound in \cite{MoldavskayaLIL} has square-root
decay as $m\alpha\uparrow1$ for fixed $d\le0$. The variational
reformulation establishes the same order for $\Lambda_m$ and
identifies a positive boundary coefficient.

\begin{theorem}\label{thm:critical}
Fix $m\ge2$ and $d\le0$, and set $\alpha_*=1/m$ and $\varepsilon=1-m\alpha$.
On the probability space with measure $\mu_d$, define
\[
 D_*:=\sup_{\|f\|_{L^p(\mu_d)}=1}
 \iint |x-y|^{-1/m} f(x)f(y)\,\mu_d(\dd x)\mu_d(\dd y),
 \qquad p=\frac{m}{m-1}.
\]
Then $0<D_*<\infty$, and as $\alpha\uparrow1/m$ through the admissible region,
\[
 \Lambda_m(\alpha,d)=C_{m,d}\sqrt{1-m\alpha}\,
             \bigl[1+O(1-m\alpha)\bigr],\qquad
 C_{m,d}:=\frac{D_*^{m/2}}{1-d}
               \sqrt{\frac{1-2d}{2m!}}>0.
\]
The constants implicit in the remainder may depend on $m,d$.
\end{theorem}
\begin{proof}
Write $\rho(x)=(1-d)x^{-d}$ on $(0,1)$ and extend it by zero to $\R$.
Thus $0\le\rho\le1-d$ because $d\le0$.
For $f,g\in L^p(\mu_d)$ set $u=\rho^{1/p}f$, $v=\rho^{1/p}g$ and extend
both by zero. Then $\|u\|_{L^p(\R)}=\|f\|_{L^p(\mu_d)}$, and the
bilinear form becomes
\[
 B_\alpha(f,g)=\iint k_\alpha(x-y)
 \rho(x)^{1/m}u(x)\rho(y)^{1/m}v(y)\dd x\dd y,
 \quad k_\alpha(t)=|t|^{-\alpha}\mathbf1_{\{|t|<1\}}.
\]
Put $r=m/2\ge1$. Young's convolution inequality, with
$1+1/m=1/p+1/r$, gives
\[
 |B_\alpha(f,g)-B_\beta(f,g)|
 \le (1-d)^{2/m}\|k_\alpha-k_\beta\|_{L^r(\R)}
       \|f\|_{L^p(\mu_d)}\|g\|_{L^p(\mu_d)}.
\]
Choose $\delta_0>0$ so that $(\alpha_*+\delta_0)r<1$. The logarithmic derivative
of $k_\alpha$ is bounded in $L^r$ uniformly for
$|\alpha-\alpha_*|<\delta_0$, since
\[
 \int_0^1 t^{-(\alpha_*+\delta_0)r}|\log t|^r\dd t<\infty.
\]
Thus $\|k_\alpha-k_{\alpha_*}\|_r=O(|\alpha-\alpha_*|)$.
The same Young bound proves finiteness of $D_*$; the test function
$f=1$ proves positivity. Taking suprema on the common unit sphere yields
\[
 D_m(\alpha,d)=D_*+O(\varepsilon).
\]
Finally, the beta-function expression for the variance normalization gives
\[
 Z_m(\alpha,d)
 =(1-d)^2\frac{2B(1-d,\varepsilon)}{1-2d+\varepsilon}
 =\frac{2(1-d)^2}{(1-2d)\varepsilon}
       \bigl[1+O(\varepsilon)\bigr].
\]
Substitution into Proposition~\ref{quad:prop:dual} proves the theorem.
\end{proof}

\begin{corollary}\label{cor:critical-upper}
The prefactor in Theorem~\ref{thm:critical} satisfies
\[
 C_{m,d}\le\frac2{\sqrt{m!}}
       \sqrt{\frac{1-2d}{1-d}},\qquad d\le0.
\]
In particular $C_{m,0}\le2/\sqrt{m!}$.
\end{corollary}
\begin{proof}
Use $Z_m\sim2(1-d)^2/[(1-2d)(1-s)]$ and the $d\le0$
formula for $\widetilde U_{m,\alpha,d}$ in Section~\ref{sec:global}.
They give
\[
 \widetilde U_{m,\alpha,d}\sim
 \frac2{\sqrt{m!}}\sqrt{\frac{1-2d}{1-d}}\sqrt{1-s}.
\]
Divide the upper inequality by $\sqrt{1-s}$ and use
Theorem~\ref{thm:critical}.
\end{proof}

\begin{remark}
The boundary coefficient is obtained by extending the deterministic
variational problem to $\alpha_*=1/m$. The rank comparison below
identifies the limit of these coefficients as $m\to\infty$.
\end{remark}

\subsection{The memory boundary of the limiting function}

For $d\le0$, the full interval $0<s<1$ is admissible. Formula
\eqref{lr:eq:Fgamma} has a finite limit $\Fm(1)$ as $s\uparrow1$.

\begin{corollary}\label{lr:cor:boundary}
For fixed $d\le0$,
\begin{equation}\label{lr:eq:boundary}
 \Xi_d(s)=\frac{e^{\Fm(1)}}\eta
             \sqrt{\frac{2\eta-1}{2}}\sqrt{1-s}\,[1+O_d(1-s)].
\end{equation}
At $d=0$ this coefficient is
\begin{equation}\label{lr:eq:Glaisher}
 \lim_{s\uparrow1}\frac{\Xi_0(s)}{\sqrt{1-s}}
 =\mathcal C_0=2^{3/2}\pi^3G(1/2)^8
 =\frac{2^{11/6}\pi e}{\mathsf A^{12}},
\end{equation}
where $\mathsf A=\exp(1/12-\zeta'(-1))$ is the
Glaisher--Kinkelin constant.
\end{corollary}

\begin{proof}
The gamma-integral expression for $\Fm(s)$ is analytic near $s=1$
when $\eta\ge1$. Moreover,
\[
 Z_d(s)=\frac{2\eta^2}{(2\eta-1)(1-s)}[1+O_d(1-s)].
\]
This proves \eqref{lr:eq:boundary}. At $\eta=1$ the Barnes expression
has $\mathfrak g(0)=\varphi(0)=0$ by continuity and gives
$\mathcal F_0(1)=\log(2\pi)+\log2-8\mathfrak g(1/2)$.
Since $\mathfrak g(1/2)=-\frac14\log\pi-\log G(1/2)$, this gives
$\mathcal F_0(1)=\log\bigl(4\pi^3G(1/2)^8\bigr)$, and division of the
exponential by $\sqrt2$ yields $2^{3/2}\pi^3G(1/2)^8$.
The classical special value
\[
 \log G(1/2)=\frac1{24}\log2-\frac14\log\pi
                         +\frac32\zeta'(-1)
\]
then yields
\[
 \mathcal F_0(1)=\log(4\pi)+\frac13\log2+1-12\log\mathsf A.
\]
Divide its exponential by $\sqrt2$ to obtain \eqref{lr:eq:Glaisher}.
\end{proof}

\subsection{Uniform control up to the boundary}
The row integral $C_s$ from Lemma~\ref{lr:lem:density} diverges
as $s\uparrow1$. A smaller H\"older exponent and interpolation
remove that divergence from the maximizing-density bound.

\begin{lemma}[A density bound uniform in memory and rank]
\label{lem:uniform-density}
Fix $d\le0$, and set
\[
 \mathsf C_d=8\,2^{7/8}\eta,\qquad \mathsf M_d=\mathsf C_d^4.
\]
For every $m\ge2$, $0<s<1$, and every maximizing density
$\rho_m$ in \eqref{lr:eq:renyi},
\[
 \|\rho_m\|_\infty\le\mathsf M_d,
 \qquad D_{p_m}(\rho_m\|\mu_d)\le\frac{16}{7}\log\mathsf C_d.
\]
\end{lemma}
\begin{proof}
The density of $\mu_d$ is at most $\eta$, so
\[
 \sup_x\int|x-y|^{-7s/8}\mu_d(\dd y)
 \le\eta\sup_x\int_0^1|x-y|^{-7/8}\dd y=\mathsf C_d.
\]
Write $\rho=\rho_m$, $D=D_{p_m}(\rho\|\mu_d)$,
$h=\log\|\rho\|_\infty$ and $c=\log\mathsf C_d$.
All are finite by Lemma~\ref{lr:lem:density} at each $s<1$,
and $D,h\ge0$. Set
\[
 r=7m/8,\qquad q=\frac r{r-1},\qquad
 \theta=\frac1{7(m-1)},\qquad \frac1q=\frac{1-\theta}{p_m}.
\]
H\"older and interpolation between $L^{p_m}$ and $L^\infty$ give
\[
 \log\|R_{s/m}\rho\|_\infty
 \le\frac{8c}{7m}+\frac{1-\theta}{m}D+\theta h,
\]
since $\log\|\rho\|_{p_m}=D/m$.
The maximizing value is nonnegative and its energy is at most
$\|R_{s/m}\rho\|_\infty$. Hence
\[
 (1+\theta)D\le\frac87c+m\theta h,
 \qquad D\le\frac87c+\frac27h,
\]
where $m\theta\le2/7$. The Euler equation
\eqref{lr:eq:eulerrho}, including $N_m\ge1$, also gives
\[
 h\le(m-1)\log\|R_{s/m}\rho\|_\infty
 \quad\Longrightarrow\quad
 h\le\frac43c+\frac76D.
\]
Combining the last two inequalities yields $h\le4c$ and
$D\le16c/7$, as required. The argument is carried out at $s<1$;
the preceding interior density bound justifies the interpolation.
\end{proof}

\begin{proof}[Proof of Theorem~\ref{thm:boundary-rank}]
Use Lemma~\ref{lem:uniform-density} in the proof of
Proposition~\ref{lr:prop:limit}. The limiting optimizing density
can likewise be chosen bounded by $\mathsf M_d$.
Since $\ell\ge0$, Lemma~\ref{lr:lem:error} gives
\[
 A_{s,d,\mathsf M_d}\le\mathcal K_d
 :=\frac{\mathsf M_d^2}{4}
       \iint \ell(x,y)^2e^{\ell(x,y)/2}\dd\mu_d^{\otimes2}<\infty.
\]
Rank monotonicity then proves \eqref{eq:uniform-rank} for the
whole interval $0<s<1$. For fixed $m$, let $s\uparrow1$ in
that inequality. Theorem~\ref{thm:critical} and
Corollary~\ref{lr:cor:boundary} identify the two limits after
dividing by $\sqrt{1-s}$, giving \eqref{eq:boundary-rank}.
The same passage in Proposition~\ref{prop:rank-transfer} proves
monotonicity of $\sqrt{m!}C_{m,d}$.
Finally combine \eqref{eq:uniform-rank} with
\[
 \frac{\Xi_d(s)}{\sqrt{1-s}}=\mathcal C_d[1+O_d(1-s)].
\]
This proves \eqref{eq:joint-boundary}. Formula
\eqref{lr:eq:Glaisher} supplies the value at $d=0$.
\end{proof}

In particular, the earlier upper bound and the identified monotone
limit give
\[
 \mathcal C_d\le\sqrt{m!}C_{m,d}
 \le2\sqrt{\frac{2\eta-1}{\eta}},\qquad d\le0.
\]
The estimate is uniform in $m$ and $s$ for each fixed $d\le0$.
The constants $\mathsf M_d$ and $\mathcal K_d$ are explicit
but have not been optimized.

\section{Finite-rank bounds accompanying the limit}\label{sec:global}
The explicit limit gives a lower bound at every finite rank.
It can be combined with elementary trial functions and upper estimates
without changing the variational problem. The formulas needed for
that comparison are recorded here; detailed uniform-ratio calculations
are in Section~S.4 of the Supplement.

Put $w(x)=x^{-d}$ and
\[
 \begin{gathered}
 g(x)=\frac{x^{1-\alpha}+(1-x)^{1-\alpha}}{1-\alpha},\\
 E=\int_0^1g(x)\dd x=\frac2{(1-\alpha)(2-\alpha)},
 \qquad F=\int_0^1w(x)g(x)^m\dd x.
 \end{gathered}
\]
Define the beta-function trial bound
\begin{equation}\label{eq:explicit-lower}
 L_{m,\alpha,d}=\frac{F}{E^{m/2}\sqrt{m!I_{m,\alpha,d}}}.
\end{equation}
Expansion of $g^m$ evaluates $F$ as a finite sum of beta functions.
For $e>\alpha/2$ set
\[
 \mathcal Z(\alpha,e)=\frac{2e^2B(e,1-\alpha)}{2e-\alpha},\qquad
 H_\tau=\frac{(1+p\tau)^{2/p}}{(1+\tau)^2}
                       \mathcal Z(\alpha,\eta(1+\tau)),\qquad
 \tau_* =\frac{\alpha(m-1)}{2-s}.
\]
The trial $f_\tau(x)=(1+p\tau)^{1/p}x^{\eta\tau}$ has
$L^p(\mu_d)$ norm one and energy $H_\tau$.
Finally put
\[
 \widetilde U_{m,\alpha,d}=
 \begin{cases}
 \displaystyle\frac{2^{s/(2\eta)}}{(1-s/(2\eta))\sqrt{m!Z_d(s)}},
                     &0\le d<1/2,\\[6pt]
 \displaystyle\frac{(2\eta)^{s/2}}{(1-s/2)\sqrt{m!Z_d(s)}},&d\le0,
 \end{cases}
\]
\begin{equation}\label{eq:weighted-sandwich}
 \begin{split}
 \underline L&=\max\left\{L_{m,\alpha,d},
               \frac{H_0^{m/2}}{\sqrt{m!Z_d(s)}},
               \frac{H_{\tau_*}^{m/2}}{\sqrt{m!Z_d(s)}}\right\},\\
 \underline L^\sharp&=\max\{\underline L,\Xi_d(s)/\sqrt{m!}\},\qquad
 \overline U=\min\{(m!)^{-1/2},\widetilde U_{m,\alpha,d}\}.
 \end{split}
\end{equation}

\begin{theorem}[Global finite-rank comparison]\label{thm:finite-bounds}
For every admissible triple,
\[
 0<L_{m,\alpha,d}<\Lambda_m(\alpha,d)<(m!)^{-1/2},
 \qquad \underline L^\sharp\le\Lambda_m(\alpha,d)\le\overline U.
\]
Moreover
\[
 \frac{\overline U}{\underline L^\sharp}\le
 C_*:=4e^{-(1+\gamma_E)/2}<1.818;
\]
for $d\ge0$ the factor can be replaced by $\sqrt3$. At $d=0$,
\[
 \overline U=\frac1{\sqrt{m!}}\min\left\{1,
       2^{s/2}\sqrt{\frac{1-s}{1-s/2}}\right\},\qquad
 \frac{\overline U}{L_{m,\alpha,0}}\le
 \varrho_m:=2\sqrt2[(1-m^{-1})(1-(2m)^{-1})]^{m/2}.
\]
The harmonic mean of $\underline L$ and $\overline U$ approximates
$\Lambda_m$ with relative error at most
\[
 \frac{\overline U-\underline L}{\overline U+\underline L}.
\]
Replacing $\underline L$ by $\underline L^\sharp$ gives the
corresponding improved bound.
\end{theorem}
\begin{proof}
The lower bound \eqref{eq:explicit-lower} is the trial $h=1$ in
\eqref{quad:eq:interval}. To see its strictness, put
$v=wg^{m-1}-(F/E)1$. Then $v\in L^2(0,1)$ and
\[
 \langle v,g\rangle=0,\qquad
 \left.\frac{d}{d\epsilon}
   \int_0^1w[K_\alpha(1+\epsilon v)]^m\dd x\right|_{\epsilon=0}
 =m\langle v,K_\alpha v\rangle>0.
\]
The energy denominator has zero first derivative.
Strict positivity follows from the Riesz energy identity and
$v\ne0$: for $d\ne0$ use its endpoint behavior, and for $d=0$
use nonconstancy of $g$. The strict upper ceiling follows from
H\"older on the product probability space: equality would require
$|x-y|^{-m\alpha}$ to factor as a product of a function of $x$
and a function of $y$, almost everywhere, which is impossible.
The maximum is attained by Proposition~\ref{quad:prop:dual}.

For the upper bound identify $[0,1]$ with the unit circle, whose
distance is no larger than $|x-y|$. Young's inequality with kernel
exponent $m/2$ gives, for Lebesgue measure,
\[
 D_m(\alpha,0)\le2^\alpha(1-m\alpha/2)^{-2/m}.
\]
Under $x=u^{1/\eta}$, the measure $\mu_d$ becomes uniform.
If $\eta\le1$, use $|u^{1/\eta}-v^{1/\eta}|
\ge|u-v|^{1/\eta}$; if $\eta\ge1$, use
$|u^{1/\eta}-v^{1/\eta}|\ge|u-v|/\eta$.
These give the two formulas for $\widetilde U$.
Direct integration gives $H_\tau$, and Corollary~\ref{lr:cor:finite}
gives the additional lower endpoint.

The uniform ratios use Jensen's inequality and the gamma-ratio
bound $\Gamma(e+1)/\Gamma(e+1-\alpha)\ge e^\alpha$
\cite[Eq.~5.6.4]{DLMF}, with the trial $\tau_*$.
Section~S.4 of the Supplement gives the full inequalities for
$\eta\le1$ and $\eta\ge1$, including the sharper comparison at
$d=0$. Increasing a lower endpoint preserves those guarantees.
For $0<a\le\lambda\le b$, the maximum relative error of
$2ab/(a+b)$ is $(b-a)/(b+a)$, attained at the endpoints.
Apply this identity to each stated pair of bounds.
\end{proof}

\begin{corollary}\label{cor:combined-lower}
For fixed admissible $(s,d)$,
\[
 0\le1-\frac{\underline L^\sharp}{\Lambda_m(s/m,d)}
 \le1-e^{-K_{s,d}/m}\le K_{s,d}/m.
\]
For $d\le0$, $K_{s,d}$ can be replaced by $\mathcal K_d$
uniformly for $0<s<1$.
\end{corollary}
\begin{proof}
Use $\underline L^\sharp\ge\Xi_d(s)/\sqrt{m!}$ in
Theorems~\ref{lr:thm:main} and \ref{thm:boundary-rank}.
\end{proof}
The improved lower endpoint is thus asymptotically exact in rank.
The power trials also remain useful for large negative $d$, where
the original trial $L_{m,\alpha,d}$ tends to zero.

\section{Local expansions at fixed Hermite rank}\label{sec:local}
\subsection{Why a local analytic branch gives the global maximum}\label{sec:branch}

Fix $m\geq2$ and $d<1/2$ throughout this section, and let $\alpha\downarrow0$ through
\eqref{quad:eq:region}. Put
\[
 \ell(x,y)=-\log|x-y|,
 \qquad R_0f=\int f\dd\mu.
\]
The expansion uses the deterministic extension at $\alpha=0$.

\begin{lemma}\label{quad:lem:branch}
For all sufficiently small positive $\alpha$, the nonnegative maximizer in
\eqref{quad:eq:D} is unique and belongs to a real analytic branch
$\alpha\mapsto(f_\alpha,D_\alpha)$ in $C[0,1]\times\R$, with
$f_0=1$ and $D_0=1$.
\end{lemma}

\begin{proof}
Choose $\varepsilon>0$ so that $m\varepsilon+\max(d,0)<1$.
The elementary integrability bound
\begin{equation}\label{quad:eq:row}
 \sup_{x\in[0,1]}\int_0^1
 |x-y|^{-m\varepsilon}(1+|\log|x-y||^k)\mu(\dd y)<\infty
 \qquad(k\geq0)
\end{equation}
follows by splitting at the diagonal and the endpoint $0$. At the common
singularity the total power is at most $m\varepsilon+\max(d,0)<1$;
logarithms are absorbed by a smaller positive power margin. The same
splitting, with the singular neighborhood first removed, gives continuity
of $x\mapsto |x-\cdot|^{-\alpha}$ in $L^m(\mu)$. H\"older's inequality
and the mean-value formula for $e^{\alpha\ell}$ therefore give
\begin{equation}\label{quad:eq:opsmall}
 R_\alpha:L^p(\mu)\to C[0,1],\qquad
 \norm{R_\alpha-R_0}_{L^p\to C}=O(\alpha).
\end{equation}
The dominated kernel series also makes $R_\alpha$ an analytic family of
bounded operators on $C[0,1]$ in a two-sided neighborhood of zero.

At zero, $B_0(f,f)=(\int f\dd\mu)^2\leq\norm f_p^2$.
For nonnegative normalized $f$, equality holds only for $f=1$.
By \eqref{quad:eq:opsmall}, $D_\alpha=1+O(\alpha)$. For any nonnegative global
maximizer $f_\alpha$,
\[
 D_\alpha=\left(\int f_\alpha\dd\mu\right)^2+O(\alpha),
 \qquad \int f_\alpha\dd\mu\longrightarrow1.
\]
The error is uniform over the unit sphere. Equation \eqref{quad:eq:opsmall}
then gives $R_\alpha f_\alpha\to1$ uniformly. By \eqref{quad:eq:Euler},
\[
 f_\alpha=\left(\frac{R_\alpha f_\alpha}{D_\alpha}\right)^{1/(p-1)}
 \longrightarrow1\quad\text{in }C[0,1].
\]
In particular every such maximizer has a continuous, strictly positive
representative for small $\alpha$.

On the open set of strictly positive continuous functions, consider
\[
 \mathcal F(\alpha,f,D)=
 \left(R_\alpha f-Df^{p-1},\ \int f^p\dd\mu-1\right).
\]
This map is analytic. With $q=p-1>0$, its derivative in $(f,D)$ at
$(0,1,1)$ is
\begin{equation}\label{quad:eq:derivative}
 (h,\delta)\longmapsto
 \left(\int h\dd\mu-qh-\delta,\ p\int h\dd\mu\right).
\end{equation}
It is invertible. Indeed, for prescribed output $(g,b_0)$, set
$c=b_0/p$, $h=c-q^{-1}(g-\int g\dd\mu)$, and
$\delta=(1-q)c-\int g\dd\mu$.
The analytic implicit-function theorem yields a unique local branch.
Every global maximizer converges uniformly to $(1,1)$ by the preceding
argument and therefore belongs to this branch. This proves the claim
about the global maximum, rather than only a stationary point.
\end{proof}

\subsection{The expansion for a general fixed weight exponent}\label{sec:quadratic}

Let $X,Y$ be independent with law $\mu_d$, and define
\begin{equation}\label{quad:eq:moments}
 a=\E \ell(X,Y),\quad b=\E \ell(X,Y)^2,\quad
 s(x)=\int \ell(x,y)\mu_d(\dd y),\quad v=\Var(s(X)).
\end{equation}
All these moments are finite. Introduce the centered two-variable remainder
\begin{equation}\label{quad:eq:V}
 r_d(x,y)=\ell(x,y)-s(x)-s(y)+a,\qquad
 V_d=\iint r_d(x,y)^2\mu_d(\dd x)\mu_d(\dd y)
     =b-a^2-2v.
\end{equation}
The last identity follows by expansion and
$\E[\ell(X,Y)\mid X]=s(X)$.

\begin{theorem}\label{quad:thm:expansion}
For fixed $m\geq2$ and $d<1/2$, as $\alpha\downarrow0$,
\begin{equation}\label{quad:eq:weighted}
 \boxed{\displaystyle
 \Lambda_m(\alpha,d)=\frac1{\sqrt{m!}}
 \left[1-\frac{m(m-1)}4V_d\alpha^2+O(\alpha^3)\right].}
\end{equation}
The remainder is $O_{m,d}(\alpha^3)$ for fixed $m,d$.
\end{theorem}

\begin{proof}
Use Lemma~\ref{quad:lem:branch} to write, in $C[0,1]$,
\[
 f_\alpha=1+\alpha u+\alpha^2z+O(\alpha^3),\qquad
 D_\alpha=1+D_1\alpha+D_2\alpha^2+O(\alpha^3).
\]
Normalization $\int f_\alpha^p\dd\mu=1$ yields
\begin{equation}\label{quad:eq:normal}
 \int u\dd\mu=0,\qquad
 \int z\dd\mu=-\frac{p-1}{2}\int u^2\dd\mu.
\end{equation}
The first-order part of $R_\alpha f_\alpha=D_\alpha f_\alpha^{p-1}$ is
$s(x)=D_1+(p-1)u(x)$. Integration and \eqref{quad:eq:normal} imply
\begin{equation}\label{quad:eq:u}
 D_1=a,\qquad u=(m-1)(s-a).
\end{equation}
Expand the objective itself, with independent $X,Y$:
\begin{align*}
 D_\alpha
 &=\E\left[e^{\alpha \ell(X,Y)}f_\alpha(X)f_\alpha(Y)\right]\\
 &=1+a\alpha+\alpha^2\left(
     \frac b2+2\int su\dd\mu+2\int z\dd\mu
       +\left(\int u\dd\mu\right)^2\right)+O(\alpha^3).
\end{align*}
Using \eqref{quad:eq:normal}--\eqref{quad:eq:u} gives
\begin{equation}\label{quad:eq:Dexp}
 D_\alpha=1+a\alpha+
 \left[\frac b2+(m-1)v\right]\alpha^2+O(\alpha^3).
\end{equation}
Similarly, from \eqref{quad:eq:Z},
\begin{equation}\label{quad:eq:Zexp}
 Z_m(\alpha,d)=1+ma\alpha+\frac{m^2b}{2}\alpha^2+O(\alpha^3).
\end{equation}
Taking logarithms in \eqref{quad:eq:exact} now gives
\begin{align*}
 \log\big(\sqrt{m!}\Lambda_m\big)
 &=\frac m2\log D_\alpha-\frac12\log Z_m\\
 &=\left\{\frac m2\left[\frac b2+(m-1)v-\frac{a^2}2\right]
               -\frac{m^2}{4}(b-a^2)\right\}\alpha^2+O(\alpha^3)\\
 &=-\frac{m(m-1)}4(b-a^2-2v)\alpha^2+O(\alpha^3).
\end{align*}
The linear term cancels exactly. Exponentiation proves \eqref{quad:eq:weighted}.
\end{proof}

\begin{remark}
$V_d>0$. Otherwise $\ell(x,y)=s(x)+s(y)-a$ almost everywhere.
On any rectangle separated from the diagonal this would force the mixed
distributional derivative of $\ell$ to vanish, whereas
$\partial_x\partial_y[-\log|x-y|]=-(x-y)^{-2}\ne0$.
Thus the quadratic correction agrees with the strict upper bound
$\Lambda_m<1/\sqrt{m!}$ in Theorem~\ref{thm:finite-bounds}.
\end{remark}

\subsubsection{Evaluation for every real weight exponent}
\begin{proposition}\label{prop:weighted-quadratic}
For $d<1$,
\[
 V_d=\frac1{12(1-d)^2}+\frac{\pi^2}{9}.
\]
In particular, for fixed $m\ge2$ and $d<1/2$, the coefficient of
$-\alpha^2$ in $\sqrt{m!}\Lambda_m$ is
\[
 C_{2,m,d}=m(m-1)\left[\frac1{48(1-d)^2}+\frac{\pi^2}{36}\right].
\]
\end{proposition}
\begin{proof}
Let $\eta=1-d>0$, take independent $X_i$ with density
$\eta x^{\eta-1}$, and put $\ell_{ij}=-\log|X_i-X_j|$.
Write $A=\Var(\ell_{12})$ and
$v=\operatorname{Cov}(\ell_{12},\ell_{13})=\Var(s(X_1))$.
Exchangeability and conditional independence give
\[
 V_d=A-2v=2A-\frac13\Var(\ell_{12}+\ell_{13}+\ell_{23}).
\]
For $n=2,3$, the cumulant generating function of
$\sum_{i<j}\ell_{ij}$ is
\[
 g_n(t)=\log\{\eta^n S_n(\eta,1,-t/2)\},
\]
where the Selberg integral is
\begin{equation}\label{eq:selberg-general}
 S_n(a,b,c)=\prod_{j=0}^{n-1}
 \frac{\Gamma(a+jc)\Gamma(b+jc)\Gamma(1+(j+1)c)}
 {\Gamma(a+b+(n+j-1)c)\Gamma(1+c)}.
\end{equation}
It is valid in a neighborhood of $t=0$; see \cite[Eq.~5.14.4]{DLMF}.
Differentiation is justified by the finite exponential moments there.
Twice differentiating the gamma product gives
\[
 A=\frac{\psi_1(\eta)+4\psi_1(1)-5\psi_1(\eta+1)}4,
\]
\[
 g_3''(0)=\frac{5\psi_1(\eta)+16\psi_1(1)-29\psi_1(\eta+1)}4.
\]
Thus $V_d=2A-g_3''(0)/3$. Use
$\psi_1(\eta)-\psi_1(\eta+1)=\eta^{-2}$ and
$\psi_1(1)=\pi^2/6$ to obtain the stated expression.
Theorem~\ref{quad:thm:expansion} gives the coefficient.
\end{proof}

\begin{corollary}\label{cor:local-weight-order}
If $d_1<d_2<1/2$ and $m\ge2$ are fixed, then
$\Lambda_m(\alpha,d_1)>\Lambda_m(\alpha,d_2)$ for all sufficiently
small positive $\alpha$.
\end{corollary}
\begin{proof}
The quadratic coefficient is strictly increasing in $d$. Its fixed,
strict difference dominates the $O(\alpha^3)$ remainders as
$\alpha\downarrow0$.
\end{proof}
The moment identity for $V_d$ holds throughout $d<1$; its use
in the LIL expansion is subject to $d<1/2$.

\subsection{A cubic formula for a general fixed weight}\label{sec:cubic}

Write $D_\alpha=D_m(\alpha,d)$, $r=m-1$ and $q=p-1=1/r$.
All expectations in this section use independent $X,Y\sim\mu_d$. Define
\begin{gather}
 a=\E \ell(X,Y),\quad b=\E \ell(X,Y)^2,\quad c=\E \ell(X,Y)^3,\notag\\
 s(x)=\E[\ell(x,Y)],\quad t(x)=s(x)-a,\quad
 v=\E t(X)^2,\quad M=\E t(X)^3,\label{cubic:eq:moments}\\
 C=\E[t(X)\ell(X,Y)^2],\qquad
 Q=\E[t(X)\ell(X,Y)t(Y)],\qquad
 \kappa_3=c-3ab+2a^3.\notag
\end{gather}
These quantities are finite; $Q$ denotes the scalar logarithmic
moment defined above.

\begin{lemma}\label{cubic:lem:cubic}
The expansion of the global maximum in \eqref{quad:eq:D} is
\begin{align}
 D_\alpha&=1+a\alpha+\left(\frac b2+rv\right)\alpha^2
                         +D_3\alpha^3+O(\alpha^4),\label{cubic:eq:Dexp}\\
 D_3&=\frac c6+rC+r^2Q-arv+\frac{r(r-1)}3M.\label{cubic:eq:D3}
\end{align}
Consequently,
\begin{align}
 \log(\sqrt{m!}\Lambda_m(\alpha,d))
  &=-\frac{m(m-1)}4(b-a^2-2v)\alpha^2
       +G_{3,m,d}\alpha^3+O(\alpha^4),\label{cubic:eq:logexp}\\
 G_{3,m,d}
  &=-\frac{m(m^2-1)}{12}\kappa_3
    +\frac{mr}{2}\left(C-2av+rQ+\frac{r-1}{3}M\right).
       \label{cubic:eq:generalC3}
\end{align}
The coefficient of $\alpha^3$ is the same before taking the logarithm,
because the linear coefficient vanishes.
\end{lemma}

\begin{proof}
Write $f_\alpha=1+\alpha u+\alpha^2z+\alpha^3h+O(\alpha^4)$ in
$C[0,1]$. Expansion of the normalization gives
\begin{align*}
 \E u&=0,\qquad \E z=-\frac q2\E u^2,\\
 \E h&=-q\E(uz)-\frac{q(q-1)}6\E u^3.
\end{align*}
The first-order Euler equation yields $u=rt$ and $D_1=a$.
In the objective, the coefficient of order three is
\[
 \frac c6+\E[u(X)\ell(X,Y)^2]
 +\E[u(X)\ell(X,Y)u(Y)]+2\E(sz)+2\E h.
\]
Since $s-qu=a$, the last two terms equal
\[
 2a\E z-\frac{q(q-1)}3\E u^3
 =-arv+\frac{r(r-1)}3M.
\]
This proves \eqref{cubic:eq:D3}; the second-order term is $b/2+rv$.
Now
\[
 [\alpha^3]\log D_\alpha=D_3-a(b/2+rv)+a^3/3,
 \qquad [\alpha^3]\log Z_m=m^3\kappa_3/6.
\]
Substitution into $(m/2)\log D_\alpha-(1/2)\log Z_m$ gives
\eqref{cubic:eq:generalC3}.
\end{proof}

\subsubsection{Connected graphs and explicit rank dependence}
Let $e_{ij}=\ell_{ij}-a$ for independent $X_i\sim\mu_d$, and define
\[
 K=\E e_{12}^3,\quad B_\star=\E(e_{12}^2e_{13}),\quad
 T_\triangle=\E(e_{12}e_{13}e_{23}),\quad
 M=\E(e_{12}e_{13}e_{14}),\quad
 P=\E(e_{12}e_{23}e_{34}).
\]
Here $K=\kappa_3$, $B_\star=C-2av$, and $P=Q$ in
Lemma~\ref{cubic:lem:cubic}. The notation $B_\star$ distinguishes this
moment from the beta function. Conditional expectation verifies these
identities, and also $M=\E t(X)^3$.

\begin{proposition}\label{prop:cubic-rank}
For fixed $m\ge2$ and $d<1/2$, write
\[
 \sqrt{m!}\Lambda_m=1-C_{2,m,d}\alpha^2-C_{3,m,d}\alpha^3
                  +O(\alpha^4),\qquad \eta=1-d.
\]
Then
\begin{equation}\label{eq:cubic-affine}
 C_{3,m,d}=m(m-1)(m\cA_d+\cB_d),
\end{equation}
where
\[
 \cA_d=\frac{\psi_2(\eta)-7\psi_2(1)+3\eta^{-3}}{96}>0,
 \qquad
 \cB_d=\frac K{12}+\frac{X_\eta}{6}-\frac{B_\star}{2}+\frac M6,
\]
and
\begin{equation}\label{eq:graph-polygamma}
 K=\psi_2(\eta)-\psi_2(1)+\frac9{4\eta^3},\qquad
 X_\eta:=M+3P=
 \frac{7\psi_2(\eta)-\psi_2(1)+15\eta^{-3}}{16}.
\end{equation}
Consequently
\[
 C_{3,m,d}=\frac{m(m-1)}2C_{3,2,d}
              +m(m-1)(m-2)\cA_d.
\]
The notation $C_{3,m,d}=-G_{3,m,d}$ fixes the sign convention.
\end{proposition}
\begin{proof}
Use $g_n$ from the preceding Selberg calculation. Expanding its third
derivative into ordered edge triples gives
\[
 g_n'''(0)=\binom n2K+18\binom n3B_\star+6\binom n3T_\triangle
                  +24\binom n4M+72\binom n4P.
\]
Every disconnected centered triple has zero expectation. A repeated
edge contributes $K$. On three vertices there are six choices of
repeated edge and adjacent edge, each with three orderings, and one
triangle with six orderings. On four vertices there are four stars
and twelve paths, each with six orderings. These counts give the
displayed coefficients. Therefore
\[
 K=g_2'''(0),\qquad
 M+3P=\frac{g_4'''(0)-4g_3'''(0)+6g_2'''(0)}{24}.
\]
Three derivatives of \eqref{eq:selberg-general} yield
\[
 g_n'''(0)=-\frac18\sum_{j=0}^{n-1}
 \left[j^3\psi_2(\eta)
 +(j^3+(j+1)^3-1)\psi_2(1)
 -(n+j-1)^3\psi_2(\eta+1)\right].
\]
Using $\psi_2(\eta+1)=\psi_2(\eta)+2\eta^{-3}$ proves
\eqref{eq:graph-polygamma}. Lemma~\ref{cubic:lem:cubic} becomes
\[
 \frac{C_{3,m,d}}{m(m-1)}=
 \frac{m+1}{12}K-\frac12\left[B_\star+(m-1)P+\frac{m-2}{3}M\right].
\]
Substitute $P=(X_\eta-M)/3$. The coefficient of $m$ is
$K/12-X_\eta/6=\cA_d$, and the remaining term is $\cB_d$.
Evaluating this affine expression at $m=2$ proves the rank identity.
Finally $\eta>1/2$, $\psi_2$ is strictly increasing, and
$\psi_2(\eta)>\psi_2(1/2)=7\psi_2(1)$. The first two terms in the
numerator of $\cA_d$ therefore have positive sum, and $3\eta^{-3}>0$;
hence $\cA_d>0$.
\end{proof}

\subsection{The unweighted coefficients and the high-rank formula}
\begin{theorem}\label{main:thm:main}
For each fixed integer $m\ge2$, as $\alpha\downarrow0$,
\begin{equation}\label{main:eq:main}
 \boxed{\displaystyle
 \Lambda_m(\alpha,0)=\frac1{\sqrt{m!}}
 \left[1-C_{2,m}\alpha^2-C_{3,m}\alpha^3+O(\alpha^4)\right],}
\end{equation}
where
\begin{align}
 C_{2,m}&=m(m-1)\left(\frac1{48}+\frac{\pi^2}{36}\right),\label{main:eq:C2}\\
 C_{3,m}&=m(m-1)\left[
 \frac{9m-1}{288}+\frac{7\pi^2}{432}
                  +\frac{3m+1}{24}\zeta(3)\right].\label{main:eq:C3}
\end{align}
Both coefficients are strictly positive. The remainder is
$O_m(\alpha^4)$ for fixed $m$.
\end{theorem}

\begin{proof}
For $d=0$, the logarithmic moments in Lemma~\ref{cubic:lem:cubic} are
\[
 a=\frac32,\quad b=\frac72,\quad c=\frac{45}{4},\quad
 v=\frac7{12}-\frac{\pi^2}{18},
\]
\[
 C=\frac{26}{9}-\frac{11\pi^2}{54}-\frac23\zeta(3),\quad
 M=\frac43-\frac{\pi^2}{72}-\zeta(3),\quad
 Q=-\frac{19}{144}+\frac{\pi^2}{216}+\frac1{12}\zeta(3).
\]
Their evaluations follow by differentiating beta integrals and summing
partial fractions; the complete calculation is Section~S.2 of the
Supplement. Substitute these values into \eqref{cubic:eq:generalC3}
and use Proposition~\ref{prop:weighted-quadratic}. Since the logarithm
has no linear term, exponentiation preserves the quadratic and cubic
coefficients.
\end{proof}

\subsection{Coefficient matching at high rank}

\begin{corollary}\label{lr:cor:Taylor}
For each fixed $d<1/2$, as $s\downarrow0$ through the admissible region,
\begin{equation}\label{lr:eq:Taylor}
 \begin{split}
 \log\Xi_d(s)={}&-
 \left(\frac1{48\eta^2}+\frac{\pi^2}{36}\right)s^2\\
 &-\frac{\psi_2(\eta)-7\psi_2(1)+3\eta^{-3}}{96}s^3+O_d(s^4),
 \end{split}
\end{equation}
where $\psi_2=(\log\Gamma)'''$. In particular,
\[
 \log\Xi_0(s)=-\left(\frac1{48}+\frac{\pi^2}{36}\right)s^2
             -\left(\frac1{32}+\frac{\zeta(3)}8\right)s^3+O(s^4).
\]
\end{corollary}

\begin{proof}
All gamma arguments in \eqref{lr:eq:Fgamma} are positive in a
neighborhood of $s=0$, so Taylor expansion under the integral
is valid. Writing $\psi=(\log\Gamma)'$, the coefficient of
$s^k$, $k\ge1$, in $\Fm$ is
\[
 \frac{(-1/2)^k}{k!(k+1)}
 \left[\psi_{k-1}(\eta)+2\psi_{k-1}(1)
             -(2^{k+1}-1)\psi_{k-1}(\eta+1)\right],
\]
where $\psi_0=\psi$. The corresponding coefficient in $\log Z_d(s)$ is
\[
 \frac{(-1)^k}{k!}
       [\psi_{k-1}(1)-\psi_{k-1}(\eta+1)]
       +\frac1{k(2\eta)^k}.
\]
Subtract one half of the latter coefficient from the former.
The constant and linear terms vanish. At orders two and three,
use
\[
 \psi_1(\eta+1)=\psi_1(\eta)-\eta^{-2},\qquad
 \psi_2(\eta+1)=\psi_2(\eta)+2\eta^{-3},
\]
along with $\psi_1(1)=\pi^2/6$ and $\psi_2(1)=-2\zeta(3)$.
This gives \eqref{lr:eq:Taylor}.
\end{proof}

\begin{corollary}[Agreement with the fixed-rank coefficients]
\label{cor:coefficient-match}
Let $C_{2,m,d}$ and $C_{3,m,d}$ be the coefficients in
Section~\ref{sec:local}. The coefficients of $-s^2$ and $-s^3$
in $\log\Xi_d(s)$ are, respectively,
\[
 \lim_{m\to\infty}\frac{C_{2,m,d}}{m^2}
 =\frac1{48\eta^2}+\frac{\pi^2}{36},\qquad
 \lim_{m\to\infty}\frac{C_{3,m,d}}{m^3}=\cA_d.
\]
\end{corollary}
\begin{proof}
Use the explicit quadratic coefficient and
$C_{3,m,d}=m(m-1)(m\cA_d+\cB_d)$ from
Proposition~\ref{prop:cubic-rank}, and compare with
\eqref{lr:eq:Taylor}. The independently proved coefficient formulas
establish the matching directly.
\end{proof}

\section{Concentrating weights and an exponential-measure limit}
\label{sec:concentration}

Fix $m\ge2$ and $0<\alpha<1/m$, and put
\[
 s=m\alpha,\qquad \eta=1-d,\qquad
 \kappa=\eta-s/2>1/2,\qquad \nu(\dd u)=e^{-u}\dd u
 \quad (u\ge0).
\]
The probability measure $\mu_d$ concentrates at $x=1$ as
$d\to-\infty$. A change of coordinates puts all the corresponding
extremal problems on the fixed space $L^p(\nu)$.

\subsection{An exact change of measure and a quadratic convergence rate}
Define, off the diagonal,
\begin{align}
 k_{\alpha,\kappa}(u,v)
 &=\left[2\kappa\sinh\!\left(\frac{|u-v|}{2\kappa}\right)\right]^{-\alpha},
 &k_{\alpha,\infty}(u,v)&=|u-v|^{-\alpha},\label{eq:hyperbolic-kernel}\\
 \mathcal D_{m,\kappa}
 &=\max_{\|h\|_{L^p(\nu)}=1}
   \iint k_{\alpha,\kappa}(u,v)h(u)h(v)\nu(\dd u)\nu(\dd v),
 &p&=\frac m{m-1},\label{eq:exponential-D}\\
 \mathcal Z_{s,\kappa}
 &=\iint k_{\alpha,\kappa}(u,v)^m\nu(\dd u)\nu(\dd v)
 =\kappa^{1-s}B(\kappa+s/2,1-s).
 \label{eq:exponential-Z}
\end{align}
The definitions of $\mathcal D_{m,\infty}$ and
$\mathcal Z_{s,\infty}$ use the limiting kernel. Every maximum exists
and has a nonnegative maximizer: the kernels belong to
$L^m(\nu\otimes\nu)$, and the compactness proof of
Proposition~\ref{quad:prop:dual} applies on this probability space.

\begin{theorem}[Limit under concentration of the weight]
\label{thm:concentration}
For every admissible $(m,\alpha,d)$,
\begin{equation}\label{eq:concentration-exact}
 \Lambda_m(\alpha,d)
 =\frac{\mathcal D_{m,\kappa}^{m/2}}
        {\sqrt{m!\mathcal Z_{s,\kappa}}}.
\end{equation}
For fixed $m,\alpha$, as $d\to-\infty$,
\begin{equation}\label{eq:concentration-limit}
 \boxed{\displaystyle
 \Lambda_m(\alpha,d)=\Lambda_m^{\exp}(\alpha)
       [1+O((1-d)^{-2})],\qquad
 \Lambda_m^{\exp}(\alpha)
 :=\frac{\mathcal D_{m,\infty}^{m/2}}
         {\sqrt{m!\Gamma(1-m\alpha)}}>0.}
\end{equation}
The linear term in $(1-d)^{-1}$ vanishes.
\end{theorem}
\begin{proof}
For $h\in L^p(\nu)$ define $f$ on $(0,1)$ by
\[
 f(e^{-u/\kappa})
 =\left(\frac\kappa\eta\right)^{1/p}
     e^{su/(2p\kappa)}h(u).
\]
Since $\mu_d(\dd x)=(\eta/\kappa)e^{-\eta u/\kappa}\dd u$,
this is a bijective isometry onto $L^p(\mu_d)$. Moreover
\[
 |e^{-u/\kappa}-e^{-v/\kappa}|^{-\alpha}
 =\kappa^\alpha e^{\alpha(u+v)/(2\kappa)}
        k_{\alpha,\kappa}(u,v).
\]
In the quadratic form, the exponential factors cancel because
$s(1-1/p)=\alpha$. Consequently
\[
 D_m(\alpha,d)=\eta^{2/m}\kappa^{\alpha-2/m}
                     \mathcal D_{m,\kappa},\qquad
 Z_m(\alpha,d)=\eta^2\kappa^{s-2}\mathcal Z_{s,\kappa}.
\]
The second equality follows either directly or from
$Z_m=\eta^2B(\eta,1-s)/\kappa$ and \eqref{eq:exponential-Z}.
Their factors cancel in \eqref{quad:eq:exact}, proving
\eqref{eq:concentration-exact}.

For independent $U,V\sim\nu$, the difference $|U-V|$ has density
$e^{-t}$ on $(0,\infty)$. Hence
\[
 \mathcal Z_{s,\infty}=\int_0^\infty e^{-t}t^{-s}\dd t
 =\Gamma(1-s).
\]
For finite $\kappa$, put $z=e^{-t/\kappa}$ in the same integral
with $t^{-s}$ replaced by $[2\kappa\sinh(t/(2\kappa))]^{-s}$.
This gives the beta formula \eqref{eq:exponential-Z}.

The elementary power-series comparison
\[
 1\le\frac{\sinh z}{z}\le e^{z^2/6}\qquad(z\ge0)
\]
follows from $(2n+1)!\ge6^n n!$. Using
$1-e^{-x}\le x$, for $t>0$ we obtain
\begin{equation}\label{eq:hyperbolic-difference}
 0\le t^{-\alpha}-[2\kappa\sinh(t/(2\kappa))]^{-\alpha}
 \le\frac{\alpha}{24\kappa^2}t^{2-\alpha}.
\end{equation}
Therefore
\begin{align}
 \|k_{\alpha,\infty}-k_{\alpha,\kappa}\|_{L^m(\nu\otimes\nu)}
 &\le\frac{\alpha}{24\kappa^2}\Gamma(2m-s+1)^{1/m},
 \label{eq:kernel-rate}\\
 0\le\mathcal D_{m,\infty}-\mathcal D_{m,\kappa}
 &\le\frac{\alpha}{24\kappa^2}\Gamma(2m-s+1)^{1/m},
 \label{eq:maximum-rate}\\
 0\le\Gamma(1-s)-\mathcal Z_{s,\kappa}
 &\le\frac{s}{24\kappa^2}\Gamma(3-s).
 \label{eq:normalization-rate}
\end{align}
For the middle bound, apply H\"older on the product space and
take suprema over the common unit sphere. Nonnegative maximizers
give its lower inequality. For the last bound use
\eqref{eq:hyperbolic-difference} with exponent $s$ and integrate.
The constant trial gives $\mathcal D_{m,\infty}\ge\Gamma(1-\alpha)>0$.
Substitution in \eqref{eq:concentration-exact}, and
$\kappa\sim1-d$, prove \eqref{eq:concentration-limit}.
\end{proof}

The exact change of measure and the even expansion of $\sinh z/z$
yield quadratic convergence of the kernels in operator norm,
which controls the variational maxima.

\begin{corollary}[An explicit finite-weight error bound]
\label{cor:concentration-error}
With $\tau_*$ from Section~\ref{sec:global}, set
\[
 G_{m,\alpha}:=\Gamma(1-\alpha)
       \frac{(1+p\tau_*)^{2/p}}{(1+\tau_*)^{2-\alpha}},\qquad
 a_\kappa:=\frac{\alpha\Gamma(2m-s+1)^{1/m}}
                    {24\kappa^2G_{m,\alpha}},\qquad
 b_\kappa:=\frac{s(1-s)(2-s)}{24\kappa^2}.
\]
Then $0<b_\kappa<1$ and, writing $x_+=\max\{x,0\}$,
\begin{equation}\label{eq:concentration-relative}
 (1-a_\kappa)_+^{m/2}
 \le\frac{\Lambda_m(\alpha,d)}{\Lambda_m^{\exp}(\alpha)}
 \le(1-b_\kappa)^{-1/2}.
\end{equation}
\end{corollary}
\begin{proof}
The normalized exponential trial
$h(u)=(1+p\tau_*)^{1/p}e^{-\tau_*u}$ has quadratic value
$G_{m,\alpha}$, so $\mathcal D_{m,\infty}\ge G_{m,\alpha}$.
Equations \eqref{eq:maximum-rate}--\eqref{eq:normalization-rate}
bound the numerator ratio between $(1-a_\kappa)_+$ and one,
and the normalization ratio between $(1-b_\kappa)$ and one.
Use \eqref{eq:concentration-exact}. Finally $\kappa>1/2$ and
$0<s<1$ give $b_\kappa<2/6<1$.
\end{proof}
The error constants depend on the fixed pair $(m,\alpha)$.

\subsection{Concentration after the high-rank limit}

As $d\to-\infty$, the probability measure $\mu_d$ concentrates near
$x=1$. The following limit is a consequence of the explicit formula
after the high-rank limit has already been taken.

\begin{corollary}\label{lr:cor:concentration}
For every fixed $0<s<1$, as $\eta=1-d\to\infty$,
\begin{equation}\label{lr:eq:concentration}
 \Xi_d(s)=\Xi_{\exp}(s)[1+O_s(\eta^{-2})],
\end{equation}
where
\begin{equation}\label{lr:eq:Xiexp}
 \Xi_{\exp}(s)=
 \frac{2\pi e^{-(1-s/2)}\Gamma(1-s/2)^2G(1-s/2)^{4/s}}
 {\sqrt{\Gamma(1-s)}}.
\end{equation}
\end{corollary}

\begin{proof}
Put $\omega=s/2$. Uniformly for $t\in[0,1]$, the gamma-ratio expansion
\[
 \log\Gamma(\eta+u)-\log\Gamma(\eta+v)
 =(u-v)\log \eta+\frac{(u-v)(u+v-1)}{2\eta}+O_s(\eta^{-2})
\]
applies to the bounded shifts in \eqref{lr:eq:Fgamma}. It gives
\[
 \Fm(s)=\omega\log \eta+2\int_0^1\log\Gamma(1-\omega t)\dd t
                         +\frac{\omega(1-\omega)}\eta+O_s(\eta^{-2}).
\]
Similarly, \eqref{lr:eq:Z} gives
\[
 \log Z_d(s)=2\omega\log \eta+\log\Gamma(1-2\omega)
                         +\frac{2\omega(1-\omega)}\eta+O_s(\eta^{-2}).
\]
Both the logarithmic term and the $\eta^{-1}$ term cancel in
$\log\Xi_d=\Fm-\tfrac12\log Z_d$. Finally, the primitive $P$
used in Proposition~\ref{lr:prop:Barnes} yields
\[
 2\int_0^1\log\Gamma(1-\omega t)\dd t
 =\log(2\pi)-1+\omega+2\log\Gamma(1-\omega)+\frac2\omega\log G(1-\omega).
\]
Exponentiation proves \eqref{lr:eq:Xiexp} and the stated rate.
\end{proof}

The subscript $\exp$ reflects the elementary fact that
$-\eta\log X$ has the standard exponential distribution when
$X$ has law $\mu_d$.

\begin{corollary}[An explicit bound for the exponential-measure constants]
\label{cor:exp-lower}
For every $m\ge2$ and $0<\alpha<1/m$,
\[
 \Lambda_m^{\exp}(\alpha)\ge
              \frac{\Xi_{\exp}(m\alpha)}{\sqrt{m!}}.
\]
\end{corollary}
\begin{proof}
Fix $m,\alpha$ in Corollary~\ref{lr:cor:finite} and let
$d\to-\infty$. Apply Theorem~\ref{thm:concentration} to its
left-hand side and Corollary~\ref{lr:cor:concentration} to its
right-hand side to obtain the stated finite-rank lower bound.
\end{proof}

\section{Scope and further questions}\label{sec:scope}
The explicit formula evaluates the normalized high-rank limit at
fixed effective memory. Theorem~\ref{thm:boundary-rank} establishes
commutation with the memory-boundary limit for each fixed $d\le0$,
with a joint remainder uniform in $m$ and $s$. For $d>0$, the
admissibility condition prevents approach to $s=1$.

At finite rank, the variational representation, global bounds and
local expansions describe $\Lambda_m(\alpha,d)$. A closed-form
evaluation of this full function remains open. The Taylor estimates
are for fixed rank and weight, whereas the weight-concentration
estimates are for fixed $(m,\alpha)$. The rate constant
$\mathcal K_d$ depends on $d$. The lower endpoint in
Corollary~\ref{cor:combined-lower} is asymptotically exact in rank;
the harmonic-mean approximation also depends on the accuracy of its
upper endpoint. The weight ordering in
Corollary~\ref{cor:local-weight-order} is local in $\alpha$;
global monotonicity in $d$ and the sign of the cubic coefficient
for general weights remain separate questions.

Further questions concern the limiting maximizing density and its
uniqueness, the first $m^{-1}$ correction to the optimized free energy,
and the high-rank limit on the exponential-measure space. Associated
Jacobi measures suggest a possible density formula whose validity
in the attractive parameter range requires a separate proof.
Commutation of the rank and weight-concentration limits is also open.

\end{document}


\maketitle
\noindent This supplement supplies the finite special-weight calculations
and detailed global comparison estimates used in the accompanying article.
References prefixed by M refer to numbered results and equations in the
main article; all other references are internal to this supplement.

\paragraph{Notation.}
Throughout,
\[
 \begin{gathered}
 m\ge2,\quad 0<\alpha<1/m,\quad s=m\alpha,\quad 2d+s<1,\\
 \eta=1-d,\quad p=m/(m-1),\quad\mu_d(\dd x)=\eta x^{\eta-1}\dd x.
 \end{gathered}
\]
The interval operators $K_\alpha$ and $R_\alpha$ have kernel
$|x-y|^{-\alpha}$ against Lebesgue measure and $\mu_d$, respectively.
The exact variational reduction in Proposition~M3.1 is
\[
 \sqrt{m!}\Lambda_m(\alpha,d)
 =D_m(\alpha,d)^{m/2}/\sqrt{Z_m(\alpha,d)},\qquad
 D_m=\max_{\|f\|_{L^p(\mu_d)}=1}\iint|x-y|^{-\alpha}f(x)f(y)\dd\mu_d^{\otimes2},
\]
\[
 Z_m(\alpha,d)=Z_d(s)=\frac{2\eta^2B(\eta,1-s)}{2\eta-s},
 \quad I_{m,\alpha,d}=Z_m/\eta^2,\quad w(x)=x^{-d}.
\]
For independent $X,Y\sim\mu_d$, set $\ell(x,y)=-\log|x-y|$,
\[
 a=\E \ell(X,Y),\quad b=\E \ell(X,Y)^2,\quad c=\E \ell(X,Y)^3,
 \quad s(x)=\E \ell(x,Y),\quad t(x)=s(x)-a,
\]
\[
 v=\E t(X)^2,\quad M=\E t(X)^3,\quad
 C=\E[t(X)\ell(X,Y)^2],\quad Q=\E[t(X)\ell(X,Y)t(Y)].
\]
For independent $X_i$ put $e_{ij}=\ell(X_i,X_j)-a$ and
\[
 K=\E e_{12}^3,\quad B_\star=\E(e_{12}^2e_{13})=C-2av,
 \quad P=\E(e_{12}e_{23}e_{34})=Q.
\]
The rank identity in Proposition~M7.7 is
\[
 C_{3,m,d}=m(m-1)(m\mathcal A_d+\mathcal B_d),\qquad
 \mathcal A_d=\frac{\psi_2(\eta)-7\psi_2(1)+3\eta^{-3}}{96},
\]
\[
 \begin{gathered}
 K=\psi_2(\eta)-\psi_2(1)+9/(4\eta^3),\qquad
 X_\eta=M+3P,\\
 X_\eta=[7\psi_2(\eta)-\psi_2(1)+15\eta^{-3}]/16,
 \quad\mathcal B_d=K/12+X_\eta/6-B_\star/2+M/6.
 \end{gathered}
\]
These formulas fix all moment conventions in the calculations below.

\section{Finite evaluations of the weighted cubic coefficient}\label{sec:evaluations}
\subsection{Integer weight parameters}
Write $H_j=\sum_{k=1}^j1/k$, $H_0=0$, and
$H_j^{(3)}=\sum_{k=1}^jk^{-3}$.

\begin{proposition}\label{prop:integer-cubic}
For each integer $N\ge1$, the coefficient at $d=1-N$ has a finite
evaluation
\[
 C_{3,m,1-N}=m(m-1)\left[
 \frac{m(2H_{N-1}^{(3)}+3N^{-3})}{96}
       +b_N+c_N\pi^2+\frac{3m+1}{24}\zeta(3)\right],
\]
where $b_N,c_N$ are rational and are computed by the beta-derivative
construction below. The first four pairs are
\[
\begin{array}{c|c|c|c}
 N&d&b_N&c_N\\ \hline
 1&0&-1/288&7/432\\
 2&-1&-17/2304&67/4320\\
 3&-2&-247/31104&2789/181440\\
 4&-3&-4043/497664&30581/1995840
\end{array}
\]
\end{proposition}
\begin{proof}
For density $Nx^{N-1}$ introduce the polynomials
\[
 P_N(x)=\sum_{j=0}^{N-1}\frac{x^j}{N-j},\qquad
 U_N(x)=2\sum_{j=0}^{N-1}\frac{H_N-H_{N-j-1}}{N-j}x^j.
\]
Splitting the row integrals at $y=x$ and expanding the integer power
of $y$ gives
\begin{align*}
 s_N(x)&=P_N(x)-x^N\log x-(1-x^N)\log(1-x),\\
 r_N(x)&:=\E \ell(x,X)^2\\
 &=U_N(x)+x^N\log^2x+(1-x^N)\log^2(1-x)\\
 &\quad-2H_Nx^N\log x+
             (2H_Nx^N-2P_N(x))\log(1-x).
\end{align*}
The mean is $a_N=H_N+1/(2N)$. Put $t_N=s_N-a_N$; then
\[
 v=N\int_0^1x^{N-1}t_N^2\dd x,\quad
 M=N\int_0^1x^{N-1}t_N^3\dd x,\quad
 B_\star=N\int_0^1x^{N-1}t_Nr_N\dd x-2a_Nv.
\]
Every term is a beta derivative:
\begin{equation}\label{eq:beta-derivatives}
 \int_0^1x^k\log^i x\log^j(1-x)\dd x
 =\left.\partial_u^i\partial_v^jB(u,v)\right|_{u=k+1,v=1},
 \qquad i+j\le3.
\end{equation}
At positive integers these derivatives reduce to rational numbers,
$\pi^2$ and $\zeta(3)$. Together with
(M7.18) and (M7.17), this is a
finite algorithm specifying $b_N,c_N$ for any prescribed $N$.

The displayed $\zeta(3)$ contribution can also be read off directly.
Only mixed logarithms of total degree three contribute it, with
coefficient twice the integral of their polynomial multiplier in
\eqref{eq:beta-derivatives}. In $M$ its coefficient is
$-6\int_0^1z(1-z)\dd z=-1$, under $z=x^N$; in $B_\star$ it is
$-4\int_0^1z(1-z)\dd z=-2/3$. Substitution into
(M7.17) gives $(3m+1)\zeta(3)/24$.
This specifies a representation, without assuming linear independence
of $1,\pi^2,\zeta(3)$ over the rationals.
\end{proof}

For example, at $d=-1$ the finite row calculation gives
\[
 v=\frac{31}{48}-\frac{\pi^2}{18},\quad
 M=\frac43-\frac{13\pi^2}{720}-\zeta(3),\quad
 B_\star=\frac{337}{288}-\frac{\pi^2}{27}-\frac23\zeta(3),
\]
and hence
\[
\begin{aligned}
 \sqrt{m!}\Lambda_m(\alpha,-1)=1
 &-m(m-1)\left(\frac1{192}+\frac{\pi^2}{36}\right)\alpha^2\\
 &-m(m-1)\left(\frac{57m-17}{2304}
 +\frac{67\pi^2}{4320}+\frac{3m+1}{24}\zeta(3)\right)\alpha^3
 +O(\alpha^4).
\end{aligned}
\]

\subsection{Rational parameters and logarithmic row identities}
For general $\eta>0$ define
\[
 s_\eta(x)=\eta\int_0^1y^{\eta-1}[-\log|x-y|]\dd y,\qquad
 r_\eta(x)=\eta\int_0^1y^{\eta-1}\log^2|x-y|\dd y.
\]
Here $r_\eta$ is a row second moment, distinct from the centered
two-variable remainder in (M7.5).

\begin{lemma}\label{lem:row-ode}
On $0<x<1$, with $b_0(x)=-\log(1-x)$,
\[
 xs_\eta'=\eta s_\eta-1-\eta b_0,\qquad
 xr_\eta'=\eta r_\eta-2s_\eta-\eta b_0^2.
\]
If $\eta=p_0/q_0$ in lowest terms, with positive integers $p_0,q_0$
and $q_0\ge2$, put $x=t^{q_0}$, $S(t)=s_\eta(t^{q_0})$ and
$N=\lfloor\eta\rfloor$. Then
\begin{equation}\label{eq:rational-row}
\begin{split}
 S(t)={}&\frac{q_0}{p_0}
 +q_0\sum_{j=1}^N\frac{t^{q_0j}}{p_0-q_0j}
 -\pi\cot(\pi\eta)t^{p_0}\\
 &-\sum_{\zeta^{q_0}=1}
       (1-\zeta^{p_0}t^{p_0})\log(1-\zeta t).
\end{split}
\end{equation}
All logarithms are the branches analytic at zero; conjugate terms pair
to a real value.
\end{lemma}
\begin{proof}
For $F_z(x)=\eta\int_0^1y^{\eta-1}|x-y|^z\dd y$,
homogeneity and integration by parts initially for $\Re z>0$ give
\[
 xF_z'(x)=(\eta+z)F_z(x)-\eta(1-x)^z.
\]
Its integrated form continues to a neighborhood of $z=0$.
Taking the first two $z$ derivatives there gives the two row
identities by differentiating the integrated analytic identity.

For noninteger $\eta$, solving the first identity by a series gives
\[
 s_\eta(x)=\eta^{-1}+A x^\eta
       +\eta\sum_{k=1}^{\infty}\frac{x^k}{k(\eta-k)}.
\]
The endpoint value is
$s_\eta(1)=\psi(\eta+1)-\psi(1)$. The digamma reflection formula
and the convergent endpoint sum give $A=-\pi\cot(\pi\eta)$.
Expanding the logarithms in \eqref{eq:rational-row} and using
$\sum_{\zeta^{q_0}=1}\zeta^j=q_0$ when $q_0$ divides $j$, and
zero otherwise, recovers precisely this series after $x=t^{q_0}$.
The finite polynomial supplies the first $N$ missing terms.
The series identity on $0<t<1$ proves the formula.
\end{proof}

The integers $p_0,q_0$ are unrelated to the Lebesgue exponent
$p=m/(m-1)$. Set
$R(t)=r_\eta(t^{q_0})$, $b_q(t)=-\log(1-t^{q_0})$, and let
$T$ have density $p_0t^{p_0-1}$ on $(0,1)$. Then
\[
 tS'=p_0S-q_0-p_0b_q,\qquad
 tR'=p_0R-2q_0S-p_0b_q^2.
\]
Writing $h=\psi(\eta+1)-\psi(1)$, we have
\[
 r_1:=R(1)=h^2+\psi_1(1)-\psi_1(\eta+1),\qquad
 a=\E S(T)=h+\frac1{2\eta}.
\]
The auxiliary primitive
\begin{equation}\label{eq:primitive-J}
 J(t):=\int_0^tu^{2p_0-1}S(u)\dd u
 =\frac{t^{2p_0}S(t)}{3p_0}
  +\frac{q_0t^{2p_0}}{6p_0^2}
  +\frac13\int_0^tu^{2p_0-1}b_q(u)\dd u
\end{equation}
follows by differentiating $t^{2p_0}S(t)$. Solving the second row
identity backwards from one and then using Tonelli gives
\begin{equation}\label{eq:eliminate-row-second}
 \E[S(T)R(T)]=p_0\left[r_1J(1)+
 \int_0^1t^{-p_0-1}J(t)\{2q_0S(t)+p_0b_q(t)^2\}\dd t\right].
\end{equation}
All integrands before expansion are nonnegative and the logarithmic
moments are finite. These identities avoid computing $R$ explicitly.
In terms of them,
\[
 M=\E(S-a)^3,\qquad
 B_\star=\E(SR)-ab-2av,
\]
where the already evaluated quantities are
\[
 b=a^2+\psi_1(1)-\psi_1(\eta)+\frac5{4\eta^2},\qquad
 v=\frac12\left[\psi_1(1)-\psi_1(\eta)+\frac7{6\eta^2}
                                -\frac{\pi^2}{9}\right].
\]

\begin{proposition}\label{prop:rational-cubic}
For every rational $\eta=1-d>1/2$, the cubic coefficient
$C_{3,m,d}$ admits a terminating reduction to a finite algebraic
combination of cyclotomic iterated-integral constants of total weight
at most three. The letters are zero and the roots of unity of order
equal to the denominator of $\eta$; $2\pi i$ is included with weight one.
\end{proposition}
\begin{proof}
The integer case is Proposition~\ref{prop:integer-cubic}. Otherwise
\eqref{eq:rational-row} makes $S$ a polynomial plus polynomial
multiples of $\log(1-\zeta t)$ and a polynomial multiple of $\pi$.
The last integral in \eqref{eq:primitive-J} is elementary:
integration by parts leaves the rational function
$t^{2p_0+q_0-1}/(1-t^{q_0})$, whose poles are roots of unity.
Thus $J$ has the same finite polynomial-logarithmic form.

Expand $\E(S-a)^3$ and \eqref{eq:eliminate-row-second} as Laurent
monomials times products of at most three logarithms. Terms are
integrated with a common endpoint cutoff. Precisely, in the integration
variable $t$, define $\operatorname{FP}\int_0^1 f(t)\dd t$ as the
coefficient of $\varepsilon^0(\log\varepsilon)^0$ in the asymptotic
expansion of $\int_\varepsilon^{1-\varepsilon}f(t)\dd t$ as
$\varepsilon\downarrow0$. The Laurent-logarithmic functions used here
admit such expansions. All terms of a sum use the same cutoff, so
finite-part extraction is linear. For an originally convergent sum,
the divergent terms cancel and the combined constant term equals its
ordinary integral. A substitution carries its cutoff endpoints with
it before the finite part is taken.

In $\E(S-a)^3$ the monomial powers are nonnegative. In
\eqref{eq:eliminate-row-second}, logarithmic terms in $J$ have powers
$0,2p_0,3p_0$; thus terms with three logarithms have monomial
exponents $-p_0-1,p_0-1,2p_0-1$, never $-1$.
A term of exponent $-1$ therefore has at most two logarithms.
For exponents other than $-1$, integration by parts reduces the
logarithmic degree by one. Partial fractions then leave simple poles
at zero and roots of unity; any higher pole is reduced by another
integration by parts, again reducing the logarithmic degree.
The procedure terminates.

Shuffle multiplication expresses the resulting primitives using
iterated integrals with those letters. The preceding degree count
bounds the total weight by three, including constants such as $\pi$
already multiplying a term. Rational-argument polygamma constants
at the endpoints belong to the same algebra. This proves the stated
finite reduction. The function class and its shuffle algebra are
standard; see \cite{RV00,ABS11}.
\end{proof}
The proposition specifies a finite evaluation class for the
rational-weight cubic Taylor coefficient.

\subsection{Two explicit noninteger examples}
At $\eta=3/2$, use $p_0=3,q_0=2$ to obtain
\[
 S(t)=\frac23+2t^2-(1-t^3)\log(1-t)-(1+t^3)\log(1+t),
 \qquad a=3-2\log2.
\]
The primitive required in \eqref{eq:primitive-J} is
\[
 \int_0^tu^5[-\log(1-u^2)]\dd u
 =\frac{(1-t^6)\log(1-t^2)+t^2+t^4/2+t^6/3}{6}.
\]
Finite logarithmic integration, detailed in
Section~\ref{app:half}, gives
\[
\begin{aligned}
 M&=\frac{761}{81}+\frac{19\pi^2}{945}-8\zeta(3),\\
 B_\star&=\frac{2053}{243}-\frac{2\pi^2}{81}-\frac{41}{6}\zeta(3),\\
 P&=-\frac{343}{486}-\frac{19\pi^2}{2835}+\frac23\zeta(3).
\end{aligned}
\]
Consequently, for every fixed $m\ge2$,
\begin{equation}\label{eq:half-expansion}
\begin{split}
 \sqrt{m!}\Lambda_m(\alpha,-\tfrac12)=1
 &-m(m-1)\left(\frac1{108}+\frac{\pi^2}{36}\right)\alpha^2\\
 &-m(m-1)\left(\frac{19m}{108}-\frac{55}{972}
 +\frac{89\pi^2}{5670}+\frac{\zeta(3)}{12}\right)\alpha^3
 +O(\alpha^4).
\end{split}
\end{equation}
All logarithms of two cancel in this coefficient. The same calculation
at $\eta=5/2$ gives
\[
 \cA_{-3/2}=\frac{7081}{40500},\qquad
 \cB_{-3/2}=-\frac{7027}{121500}
       +\frac{12503\pi^2}{810810}+\frac{\zeta(3)}{12}.
\]
These examples illustrate the finite reduction at two half-integer
weight parameters.

\section{Logarithmic integrations at \texorpdfstring{$d=0$}{d=0}}\label{app:uniform}
\subsection{Quadratic logarithmic moments}

Here $\mu_0$ is uniform on $[0,1]$. Direct integration gives
\begin{equation}\label{quad:eq:s0}
 s(x)=1-x\log x-(1-x)\log(1-x),\qquad a=\frac32,
\end{equation}
with $0\log0=0$. For example, split the defining integral for $s$ at $y=x$
and use $\int_0^t-\log u\dd u=t(1-\log t)$.
The density of $|X-Y|$ is $2(1-t)$ for $0<t<1$, so
\[
 b=2\int_0^1(1-t)\log^2t\dd t=\frac72,
 \qquad b-a^2=\frac54.
\]
To calculate $v$, write
$H(x)=-x\log x-(1-x)\log(1-x)$; then $s=1+H$ and
$\int H\dd x=1/2$. Squaring $H$ gives
\begin{equation}\label{quad:eq:H2}
 \int_0^1H(x)^2\dd x=
 2\int_0^1x^2\log^2x\dd x+
 2\int_0^1x(1-x)\log x\log(1-x)\dd x.
\end{equation}
The first integral is $2/27$. For the mixed integral differentiate
$B(a_0,b_0)=\int_0^1x^{a_0-1}(1-x)^{b_0-1}\dd x$:
\[
 \partial_{a_0}\partial_{b_0}B(a_0,b_0)
 =B(a_0,b_0)\left[
  (\psi(a_0)-\psi(a_0+b_0))(\psi(b_0)-\psi(a_0+b_0))
  -\psi_1(a_0+b_0)\right].
\]
At $a_0=b_0=2$, use $B(2,2)=1/6$,
$\psi(2)-\psi(4)=-5/6$ and
$\psi_1(4)=\pi^2/6-49/36$. Consequently
\[
 \int_0^1x(1-x)\log x\log(1-x)\dd x
   =\frac{37}{108}-\frac{\pi^2}{36}.
\]
Equation \eqref{quad:eq:H2} therefore yields
\[
 \int H^2\dd x=\frac56-\frac{\pi^2}{18},\qquad
 v=\frac7{12}-\frac{\pi^2}{18},\qquad
 V_0=\frac54-2v=\frac1{12}+\frac{\pi^2}{9}.
\]
Substitution into Theorem~M7.2 proves
\begin{equation}\label{quad:eq:unweighted}
 \boxed{\displaystyle
 \Lambda_m(\alpha,0)=\frac1{\sqrt{m!}}
 \left[1-m(m-1)\left(\frac1{48}+\frac{\pi^2}{36}\right)
          \alpha^2+O(\alpha^3)\right].}
\end{equation}
This formula holds for each fixed integer $m\geq2$; rank one instead has
$\Lambda_1=1$ exactly under the variance normalization.

\subsection{Cubic logarithmic moments}

Now $X,Y$ are uniform. Put
\[
 H(x)=-x\log x-(1-x)\log(1-x),\quad
 s(x)=1+H(x),\quad t(x)=H(x)-\tfrac12.
\]
The elementary moments and the quadratic calculation give
\begin{equation}\label{cubic:eq:basic}
 a=\frac32,\quad b=\frac72,\quad c=\frac{45}{4},\quad
 \kappa_3=\frac94,\quad v=\frac7{12}-\frac{\pi^2}{18}.
\end{equation}
For example $\E\ell^j=2\int_0^1(1-u)(-\log u)^j\dd u$.
The remaining moments have the values
\begin{align}
 C&=\frac{26}{9}-\frac{11\pi^2}{54}-\frac23\zeta(3),\label{cubic:eq:Cval}\\
 M&=\frac43-\frac{\pi^2}{72}-\zeta(3),\label{cubic:eq:Mval}\\
 Q&=-\frac{19}{144}+\frac{\pi^2}{216}+\frac1{12}\zeta(3).
     \label{cubic:eq:Qval}
\end{align}

To prove \eqref{cubic:eq:Cval}--\eqref{cubic:eq:Mval}, set
$J(x)=x\log^2x+(1-x)\log^2(1-x)$. Then
\[
 \int \ell(x,y)^2\dd y=2+2H(x)+J(x),\qquad
 C=2v+\int HJ\dd x-\frac14.
\]
Differentiation of the beta integral gives
\begin{align*}
 \int HJ\dd x
   &=\frac{71}{36}-\frac{5\pi^2}{54}-\frac23\zeta(3),\\
 \int H^3\dd x
   &=\frac73-\frac{7\pi^2}{72}-\zeta(3),\qquad
 \int H^2\dd x=\frac56-\frac{\pi^2}{18}.
\end{align*}
More explicitly, the first two expressions are
$4/27-2\partial_a\partial_b^2B(2,2)$ and
$3/64-6\partial_a^2\partial_bB(3,2)$, respectively.
Using $M=\int H^3-(3/2)\int H^2+1/4$ proves the claims.

\subsubsection{The double logarithmic moment}
Define
\begin{align*}
 I_A&=\iint xy\log x\log y\,[-\log|x-y|]\dd x\dd y,\\
 I_B&=\iint x(1-y)\log x\log(1-y)\,[-\log|x-y|]\dd x\dd y.
\end{align*}
Reflection and exchange of variables show that
\begin{equation}\label{cubic:eq:QAB}
 Q=2(I_A+I_B)-\frac{23}{24}+\frac{\pi^2}{18}.
\end{equation}
The first integral is obtained by differentiating the identity
\[
 \iint x^a y^b|x-y|^g\dd x\dd y
 =\frac{B(a+1,g+1)+B(b+1,g+1)}{a+b+g+2}
\]
once in each of $a,b,g$, with a minus sign, at $(1,1,0)$. Hence
\begin{equation}\label{cubic:eq:IA}
 I_A=\frac{27}{128}-\frac{\pi^2}{96}.
\end{equation}

For $I_B$, split the square at $y=x$ and reflect the upper triangle by
exchanging $x,y$. The two terms are
\begin{align*}
 A_\triangle&=-\int_0^1x\log x\int_0^x
        (1-y)\log(1-y)\log(x-y)\dd y\dd x,\\
 B_\triangle&=-\int_0^1(1-x)\log(1-x)\int_0^x
        y\log y\log(x-y)\dd y\dd x.
\end{align*}
For the second term, the inner integral equals
\[
 x^2\left[\frac12\log^2x-\log x+1-\frac{\pi^2}{12}\right].
\]
Beta derivatives therefore give
\begin{equation}\label{cubic:eq:Btri}
 B_\triangle=\frac{3125}{6912}-\frac{17\pi^2}{576}
                                      -\frac{\zeta(3)}{12}.
\end{equation}
For the first term, expand
\[
 (1-y)\log(1-y)=-y+\sum_{n=2}^\infty\frac{y^n}{n(n-1)}.
\]
All terms of the series are nonnegative, so Tonelli applies to the series
part after accounting for the signs of the two logarithms. Since
\[
 \int_0^x y^n\log(x-y)\dd y
 =\frac{x^{n+1}}{n+1}(\log x-H_{n+1}),
\]
where $H_k=\sum_{j=1}^k1/j$, we obtain
\begin{align}
 A_\triangle&=\frac1{16}-S_R-S_H,\notag\\
 S_R&=\sum_{k=3}^\infty
       \frac{2}{(k-2)(k-1)k(k+2)^3},\label{cubic:eq:series}\\
 S_H&=\sum_{k=3}^\infty
       \frac{H_k}{(k-2)(k-1)k(k+2)^2}.\notag
\end{align}
These sums evaluate to
\begin{align}
 S_R&=\frac{1741}{6912}-\frac{13\pi^2}{864}-\frac{\zeta(3)}{12},\notag\\
 S_H&=\frac{209}{3456}+\frac{\pi^2}{1728}-\frac{\zeta(3)}{24}.
      \label{cubic:eq:sumvals}
\end{align}
For a reproducible derivation, partial fractions give
\begin{align*}
 \frac{2}{(k-2)(k-1)k(k+2)^3}
 &=\frac1{64(k-2)}-\frac2{27(k-1)}+\frac1{8k}
    -\frac{115}{1728(k+2)}\\
 &\hspace{8mm}-\frac{13}{144(k+2)^2}-\frac1{12(k+2)^3},\\
 \frac1{(k-2)(k-1)k(k+2)^2}
 &=\frac1{32(k-2)}-\frac1{9(k-1)}+\frac1{8k}
    -\frac{13}{288(k+2)}-\frac1{24(k+2)^2}.
\end{align*}
The simple-pole coefficients in each line sum to zero, so divergent terms
are combined before limits are taken. Summation by parts gives
\begin{align*}
 \sum_{k=3}^\infty H_k\left(\frac1{k-1}-\frac1k\right)&=\frac54,\\
 \sum_{k=3}^\infty H_k\left(\frac1{k-2}-\frac1k\right)&=\frac72,\\
 \sum_{k=3}^\infty H_k\left(\frac1{k+2}-\frac1k\right)&=\frac1{24}-\zeta(2),\\
 \sum_{k=3}^\infty\frac{H_k}{(k+2)^2}
     &=\zeta(3)+\zeta(2)-\frac{635}{288}.
\end{align*}

For the last identity, shift the summation index and use the classical
Euler sum
\[
 \sum_{k\ge1}\frac{H_k}{k^2}=2\zeta(3).
\]
These identities prove \eqref{cubic:eq:sumvals} without fitting decimal values.
Combining \eqref{cubic:eq:Btri}--\eqref{cubic:eq:sumvals} yields
\begin{equation}\label{cubic:eq:IB}
 I_B=\frac{233}{1152}-\frac{13\pi^2}{864}+\frac{\zeta(3)}{24}.
\end{equation}
Now \eqref{cubic:eq:QAB} and \eqref{cubic:eq:IA} prove \eqref{cubic:eq:Qval}.

\subsection{Completion and an independent rank-two check}

Substituting \eqref{cubic:eq:basic}--\eqref{cubic:eq:Qval} into
(M7.16) simplifies to
\[
 G_{3,m,0}=-m(m-1)
 \left[\frac{9m-1}{288}+\frac{7\pi^2}{432}
                       +\frac{3m+1}{24}\zeta(3)\right].
\]
The quadratic coefficient is (M7.20). Because the logarithm has no
linear term, exponentiation changes neither coefficient through order
three, proving Theorem~M7.8.

For $m=2$, $D_\alpha$ is the largest eigenvalue of
$K_\alpha f(x)=\int_0^1|x-y|^{-\alpha}f(y)\dd y$.
At zero, $K_0$ is the projection onto the constant functions. The third
Rayleigh--Schr\"odinger coefficient of this isolated eigenvalue is
\[
 \frac c6+C+Q-av,
\]
because the reduced resolvent on the orthogonal complement of constants
is the identity. This is exactly (M7.14) at $r=1$. Explicitly,
\begin{align*}
 \lambda_{\max}(K_\alpha)
 &=1+\frac32\alpha+\left(\frac73-\frac{\pi^2}{18}\right)\alpha^2\\
 &\quad+\left(\frac{541}{144}-\frac{25\pi^2}{216}
                         -\frac7{12}\zeta(3)\right)\alpha^3
 +O(\alpha^4).
\end{align*}
Using $\Lambda_2=\sqrt{(1-2\alpha)(1-\alpha)/2}\,
\lambda_{\max}(K_\alpha)$ gives
\[
 \Lambda_2=\frac1{\sqrt2}\left[
 1-\left(\frac1{24}+\frac{\pi^2}{18}\right)\alpha^2
 -\left(\frac{17}{144}+\frac{7\pi^2}{216}
                   +\frac7{12}\zeta(3)\right)\alpha^3
 +O(\alpha^4)\right].
\]

\section{Finite logarithmic integration for half-integer weights}\label{app:half}
The reduction in Section~\ref{sec:evaluations} becomes particularly
explicit when $q_0=2$ and $p_0$ is odd. Write
$a(t)=\log(1-t)$, $b(t)=\log(1+t)$, $\ell=\log2$ and $z=\zeta(3)$.
After inserting \eqref{eq:rational-row} and
\eqref{eq:primitive-J} into the moments, the integrals to evaluate are
finite linear combinations of
\[
 I_{n,i,j}=\operatorname{FP}\int_0^1t^na(t)^ib(t)^j\dd t,
 \qquad n\in\mathbb Z,\quad i+j\le3.
\]
The finite part refers to a common endpoint cutoff in a sum whose
original integral is convergent. In particular it preserves the
cancellations at zero caused by $J(t)=O(t^{2p_0})$.

For $n\ne-1$ and $i+j>0$, integration by parts using
$(t^{n+1}-1)/(n+1)$ gives the recurrence
\[
\begin{split}
 I_{n,i,j}={}&-\frac{c_{n,i,j}}{n+1}
 +\frac{i}{n+1}\operatorname{FP}\int_0^1
       \frac{t^{n+1}-1}{1-t}a^{i-1}b^j\dd t\\
 &-\frac{j}{n+1}\operatorname{FP}\int_0^1
       \frac{t^{n+1}-1}{1+t}a^ib^{j-1}\dd t,
\end{split}
\]
where $c_{n,i,j}$ is the coefficient of $t^{-n-1}$ in the Taylor
series of $a(t)^ib(t)^j$ if $n<-1$, and is zero otherwise.
The boundary at one vanishes because of the subtraction in the
primitive. Polynomial division and partial fractions reduce the
logarithmic degree, so the recurrence terminates. For degree zero,
$I_{n,0,0}=1/(n+1)$ when $n\ne-1$, with finite-part interpretation
if necessary, and $I_{-1,0,0}=0$.

The two-logarithm base integrals are
\begin{align*}
 \int_0^1\frac{a^2}{t}\dd t&=2z,&
 \int_0^1\frac{ab}{t}\dd t&=-\frac58z,&
 \int_0^1\frac{b^2}{t}\dd t&=\frac14z,\\
 \int_0^1\frac{a^2}{1+t}\dd t
 &=\frac74z-\frac{\pi^2\ell}{6}+\frac{\ell^3}{3},&&&\\
 \int_0^1\frac{ab}{1+t}\dd t
 &=\frac18z-\frac{\pi^2\ell}{12}+\frac{\ell^3}{3},&
 \int_0^1\frac{b^2}{1+t}\dd t&=\frac{\ell^3}{3}.&&
\end{align*}
The one-logarithm values are
\[
 \int_0^1\frac a t\dd t=-\frac{\pi^2}{6},\quad
 \int_0^1\frac b t\dd t=\frac{\pi^2}{12},\quad
 \int_0^1\frac a{1+t}\dd t=-\frac{\pi^2}{12}+\frac{\ell^2}{2},
 \quad \int_0^1\frac b{1+t}\dd t=\frac{\ell^2}{2}.
\]
Together with $\int_0^1(1+t)^{-1}\dd t=\ell$, these are a complete
set of base values for the recurrence at total logarithmic degree
at most three occurring here. They follow by series integration,
integration by parts and the substitutions $u=(1-t)/2$ or $u=1/(1+t)$;
the needed endpoint values are
\[
 \operatorname{Li}_2(1/2)=\frac{\pi^2}{12}-\frac{\ell^2}{2},\qquad
 \operatorname{Li}_3(1/2)=\frac78z-\frac{\pi^2\ell}{12}
                         +\frac{\ell^3}{6}.
\]
Alternatively they can be evaluated in the harmonic-polylogarithm
algebra of \cite{RV00}. The recurrence uses exact rational arithmetic;
no integer-relation fitting is required.

For $p_0=3$ the input quantities are
\[
\begin{split}
 R(1)&=(8/3-2\ell)^2+40/9-\pi^2/3,\\
 J(t)&=\frac{t^6S(t)}9+\frac{t^6}{27}
       +\frac{(1-t^6)\log(1-t^2)+t^2+t^4/2+t^6/3}{18}.
\end{split}
\]
Applying the displayed recurrence to $3\int t^2(S-a)^3\dd t$ and
\eqref{eq:eliminate-row-second} yields the moments preceding
\eqref{eq:half-expansion}. The value of $P$ also follows from
$(X_{3/2}-M)/3$, providing an independent algebraic route.
For $p_0=5$ the same recurrence yields
\[
\begin{aligned}
 M&=\frac{91864}{10125}+\frac{6497\pi^2}{135135}-8\zeta(3),\\
 B_\star&=\frac{28246}{3375}-\frac{2\pi^2}{135}-\frac{41}{6}\zeta(3),\\
 P&=-\frac{35513}{60750}-\frac{6497\pi^2}{405405}+\frac23\zeta(3).
\end{aligned}
\]
Substitution into (M7.17) proves the stated
$\eta=5/2$ evaluation.

\section{Detailed global comparison estimates}\label{sec:global}
\begin{theorem}[Global bounds and relative approximation]\label{main:thm:global}
Let $m\ge2$, $0<\alpha<1/m$, and $s=m\alpha$. Put
\[
 U_m(\alpha)=\frac1{\sqrt{m!}}\min\left\{1,
 2^{s/2}\sqrt{\frac{1-s}{1-s/2}}\right\},\qquad
 \varrho_m=2\sqrt2\left[(1-m^{-1})(1-(2m)^{-1})\right]^{m/2}.
\]
Then
\[
 0<L_{m,\alpha,0}<\Lambda_m(\alpha,0)\le U_m(\alpha),
 \qquad \frac{U_m(\alpha)}{L_{m,\alpha,0}}
 \le\varrho_m<2\sqrt2e^{-3/4}.
\]
Writing $L=L_{m,\alpha,0}$ and $U=U_m(\alpha)$, the explicit approximation
$\widehat\Lambda_m=2LU/(L+U)$ satisfies
\[
 \left|\frac{\widehat\Lambda_m}{\Lambda_m(\alpha,0)}-1\right|
 \le\frac{U-L}{U+L}\le\frac{\varrho_m-1}{\varrho_m+1}.
\]
The last bound is less than $2.944\%$ for $m=2$, less than
$7.886\%$ for $m=3$, and less than $14.386\%$ for every $m\ge2$.
All these estimates are uniform for $0<\alpha<1/m$.
\end{theorem}

The theorem follows from the exact positive variational formulation,
the periodic form of Young's inequality, and Jensen's inequality.
Section~\ref{sec:global} proves it and gives weighted upper bounds for
every admissible real $d$. The original lower bound degenerates when
$d\to-\infty$, so a uniform relative comparison over all weights
requires an additional trial. A normalized power provides one explicitly.

\begin{theorem}[Uniform comparison over all weights and ranks]
\label{main:thm:weighted-global}
Let $m\ge2$, $0<\alpha<1/m$ and $2d+m\alpha<1$. The explicit
beta-function and power bounds $\underline L$ and $\overline U$
defined in \eqref{eq:weighted-sandwich} satisfy
\[
 0<\underline L\le\Lambda_m(\alpha,d)\le\overline U,
 \qquad
 1\le\frac{\overline U}{\underline L}\le
 C_*:=4\exp\left(-\frac{1+\gamma_E}{2}\right)<1.818,
\]
where $\gamma_E$ is Euler's constant. Consequently
\[
 \widehat\Lambda_m^{\mathrm w}
 :=\frac{2\underline L\,\overline U}{\underline L+\overline U},
 \qquad
 \left|\frac{\widehat\Lambda_m^{\mathrm w}}{\Lambda_m}-1\right|
 \le\frac{\overline U-\underline L}{\overline U+\underline L}
 \le\frac{C_*-1}{C_*+1}<0.29026.
\]
For $d\ge0$ the comparison factor can be replaced by $\sqrt3$.
For $d=0$ the sharper rank-dependent guarantees of
Theorem~\ref{main:thm:global} also hold for this new approximation.
\end{theorem}

This comparison is uniform even as $d\to-\infty$ and as the rank
increases. It estimates the deterministic constant itself; the LIL
continues to require the stochastic hypotheses of Section~M2.
The constant $C_*$ gives an explicit uniform approximation guarantee.

\subsection{The explicit lower bound and strictness}
Define
\[
 \begin{gathered}
 g(x)=K_\alpha\mathbf1(x)=
 \frac{x^{1-\alpha}+(1-x)^{1-\alpha}}{1-\alpha},\\
 E=\int_0^1g(x)\dd x=\frac2{(1-\alpha)(2-\alpha)},\qquad
 F=\int_0^1w(x)g(x)^m\dd x.
 \end{gathered}
\]
The trial function $h=1$ in (M3.1) gives the lower
bound of \cite{MoldavskayaLIL}:
\begin{equation}\label{eq:explicit-lower}
 \begin{split}
 L_{m,\alpha,d}&:=\frac{F}{E^{m/2}\sqrt{m!I_{m,\alpha,d}}}\\
 &=\left(\frac{2-\alpha}{2(1-\alpha)}\right)^{m/2}
 \frac{\displaystyle\sum_{k=0}^m\binom mk
 B(1-d+(1-\alpha)k,1+(1-\alpha)(m-k))}
 {\sqrt{m!I_{m,\alpha,d}}}.
 \end{split}
\end{equation}
The equality follows by expanding $g^m$ and using the beta integral.

\begin{proposition}\label{prop:strict}
For every parameter triple in (M2.1),
\[
 0<L_{m,\alpha,d}<\Lambda_m(\alpha,d)<\frac1{\sqrt{m!}}.
\]
\end{proposition}
\begin{proof}
All beta arguments in \eqref{eq:explicit-lower} are positive,
so $L_{m,\alpha,d}>0$.
For strictness set $b_0=wg^{m-1}$ and $v_0=b_0-(F/E)\mathbf1$.
These functions belong to $L^2$, because $d<1/2$ and $g$ is bounded.
For $h_\varepsilon=1+\varepsilon v_0$,
\[
 \left.\frac{d}{d\varepsilon}
    \langle h_\varepsilon,K_\alpha h_\varepsilon\rangle
 \right|_{0}=2\langle v_0,g\rangle=0,
\]
and
\[
 \left.\frac{d}{d\varepsilon}
   \int w(K_\alpha h_\varepsilon)^m\dd x\right|_{0}
 =m\langle b_0,K_\alpha v_0\rangle
 =m\langle v_0,K_\alpha v_0\rangle>0.
\]
The last inequality uses strict positive definiteness of the Riesz
energy for nonzero real $L^2$ functions extended by zero. Here $v_0\ne0$:
if $d>0$, $b_0$ diverges at zero; if $d<0$, it tends to zero while
remaining positive inside; if $d=0$, $g$ is nonconstant for $\alpha>0$.
Thus a sufficiently small positive perturbation strictly increases the
quotient. Differentiation is justified by
$K_\alpha:L^2\to L^\infty$ and $w\in L^1$.

For the upper bound take a nonnegative maximizer $f$ in
Proposition~M3.1. H\"older's inequality on the product
probability space gives $D_m\le Z_m^{1/m}$. Equality would force
\[
 |x-y|^{-m\alpha}=c f(x)^p f(y)^p\quad\text{almost everywhere}
\]
for a positive constant $c$. This factorization is impossible. For
almost every two distinct interior points $x_1,x_2$, it would make
$|x_1-y|/|x_2-y|$ constant almost everywhere in $y$, whereas this ratio
is nonconstant on an interval separated from both points. Thus
$D_m<Z_m^{1/m}$, which gives the strict upper inequality.
\end{proof}

\subsection{The periodic Young bound}
\begin{proposition}\label{prop:young}
For $d=0$, the upper bound $U_m$ and the comparison factor $\varrho_m$
of Theorem~\ref{main:thm:global} satisfy
\[
 \Lambda_m(\alpha,0)\le U_m(\alpha),\qquad
 \frac{U_m(\alpha)}{L_{m,\alpha,0}}\le\varrho_m,
 \qquad
 U_m(\alpha)\sim\frac2{\sqrt{m!}}\sqrt{1-m\alpha}.
\]
\end{proposition}
\begin{proof}
Identify $[0,1)$ with the unit circle $\mathbb T$ with probability
Haar measure. Its distance satisfies
$d_{\mathbb T}(x,y)\le|x-y|$, so the interval kernel is dominated by
$k(x-y)=d_{\mathbb T}(x,y)^{-\alpha}$. Young's convolution inequality
with $r=m/2$, $p=m/(m-1)$ and $1+1/m=1/p+1/r$ gives
\[
 D_m(\alpha,0)\le\|k\|_{L^r(\mathbb T)}
 =\left(2\int_0^{1/2}t^{-\alpha r}\dd t\right)^{1/r}
 =2^\alpha(1-m\alpha/2)^{-2/m}.
\]
This is classical Young's inequality with constant one; see, for
example, \cite[Example 2.1(1)]{Satomi}. Since
$Z_m=[(1-s)(1-s/2)]^{-1}$, Proposition~M3.1 and
the normalization ceiling give $\Lambda_m\le U_m$.

For the comparison, Jensen's inequality gives $F=\int g^m\ge E^m$
when $d=0$. Therefore
\[
 \frac{U_m}{L_{m,\alpha,0}}\le Q_m(s):=
 \frac{2^{s/2}}{1-s/2}
 \left[(1-s/m)(1-s/(2m))\right]^{m/2}.
\]
For $0<s<1$ and $m\ge2$,
\[
 \frac{d}{ds}\log Q_m(s)
 =\frac{\log2}{2}+\frac1{2-s}
 -\frac12\left(\frac1{1-s/m}+\frac1{2-s/m}\right)
 \ge\frac{\log2}{2}-\frac1{4-s}>0.
\]
The inequality follows by replacing $s/m$ with $s/2$ in the two
subtracted terms. Thus $Q_m(s)\le Q_m(1)=\varrho_m$.
Using $\log(1-x)<-x$ in its two factors yields
$\varrho_m<2\sqrt2e^{-3/4}$. The asymptotic of $U_m$ is immediate.
\end{proof}

\begin{corollary}\label{cor:relative}
The harmonic-mean approximation in Theorem~\ref{main:thm:global}
has the stated relative error. Among constants $a>0$ used to
approximate an unknown $\lambda\in[L,U]$, $a=2LU/(L+U)$ minimizes
$\sup_{\lambda\in[L,U]}|a/\lambda-1|$.
\end{corollary}
\begin{proof}
The endpoint errors are $a/L-1$ and $1-a/U$ when $L\le a\le U$;
outside that interval projection onto it improves the error. Equating
the two endpoint errors gives $a=2LU/(L+U)$ and worst error
$(U-L)/(U+L)$. Apply Proposition~\ref{prop:young}. In particular
\[
 \varrho_2=\frac{3\sqrt2}{4},\qquad
 \varrho_3=\frac{10\sqrt{10}}{27},\qquad
 \sup_{m\ge2}\frac{\varrho_m-1}{\varrho_m+1}
 \le\frac{2\sqrt2e^{-3/4}-1}{2\sqrt2e^{-3/4}+1}<0.14386.
\]
This proves Theorem~\ref{main:thm:global}.
\end{proof}

The minimum defining $U_m$ is essential: its second entry improves
the ceiling only after the positive solution of
$2^s(1-s)=1-s/2$, approximately $s=0.3663732586$. The lower bound
and $U_m$ both have square-root order at $s=1$. The limiting
prefactor of $\Lambda_m$ is identified in Section~M5.

\subsection{Weighted upper bounds and their scope}
\begin{proposition}\label{prop:weighted-upper}
For every triple in (M2.1), with $\eta=1-d$ and
$s=m\alpha$, let
\[
 Z_m=\frac{2\eta^2B(\eta,1-s)}{2\eta-s},\qquad
 \widetilde U_{m,\alpha,d}=
 \begin{cases}
 \displaystyle\frac{2^{s/(2\eta)}}
 {(1-s/(2\eta))\sqrt{m!Z_m}},&0\le d<1/2,\\[6pt]
 \displaystyle\frac{(2\eta)^{s/2}}
 {(1-s/2)\sqrt{m!Z_m}},&d\le0.
 \end{cases}
\]
Then
\[
 \Lambda_m(\alpha,d)\le
 \min\{(m!)^{-1/2},\widetilde U_{m,\alpha,d}\}.
\]
The two formulas agree at $d=0$.
\end{proposition}
\begin{proof}
Under $x=u^{1/\eta}$ the measure $\mu_d$ becomes uniform and
$L^p$ norms are preserved. If $\eta\le1$, convexity gives
$|u^{1/\eta}-v^{1/\eta}|\ge|u-v|^{1/\eta}$.
Apply the preceding Young bound with exponent $\alpha/\eta$.
Its integrability condition holds because $s<2\eta-1\le\eta$.
If $\eta\ge1$, the derivative of $u^{1/\eta}$ is at least $1/\eta$
on $(0,1)$, and hence
$|u^{1/\eta}-v^{1/\eta}|\ge\eta^{-1}|u-v|$.
The same bound yields
$D_m\le\eta^\alpha2^\alpha(1-s/2)^{-2/m}$.
Use (M3.5), keeping the original weighted $Z_m$.
\end{proof}

For $d\ge0$ the raw bound also equals
\[
 \widetilde U_{m,\alpha,d}
 =\left[\frac{2^{1+s/\eta}}
 {m!(2\eta-s)B(\eta,1-s)}\right]^{1/2}.
\]
For example, at $(m,\alpha,d)=(2,0.2,0.1)$ its value is
$0.7357244792$, so the minimum remains $1/\sqrt2$; at
$(2,0.45,0.02)$ it is $0.4215987397$ and improves the ceiling.
These examples illustrate why the normalization ceiling must be kept.
A uniform relative-error conclusion over all weights also needs a lower
bound that remains positive as $d\to-\infty$.

\begin{proposition}\label{prop:negative-d}
Fix $m\ge2$ and $0<\alpha<1/m$. As $\eta=1-d\to\infty$,
\[
 L_{m,\alpha,d}\longrightarrow0,\qquad
 \liminf_{d\to-\infty}\Lambda_m(\alpha,d)
 \ge\frac1{\sqrt{m!\Gamma(1-m\alpha)}}>0.
\]
Consequently a uniform relative comparison with this lower bound
cannot hold over all negative $d$.
\end{proposition}
\begin{proof}
With $s=m\alpha$ fixed, the beta ratio gives
$Z_m\sim\Gamma(1-s)\eta^s$. If $X$ has density
$\eta x^{\eta-1}$ then $X\to1$ in probability; since $g$ is bounded,
$\E g(X)^m\to g(1)^m$. Formula \eqref{eq:explicit-lower}, with
$I=W^2Z_m$, gives
\[
 L_{m,\alpha,d}\sim
 \frac{[g(1)/\sqrt E]^m}{\sqrt{m!\Gamma(1-s)}}\eta^{-s/2}.
\]
On the other hand the test $f=1$ and Jensen's inequality give
$D_m\ge(\E|X-Y|)^{-\alpha}$ for independent $X,Y\sim\mu_d$.
Direct integration yields
\[
 \E|X-Y|=\frac{2\eta}{(\eta+1)(2\eta+1)}\le\eta^{-1}.
\]
Thus $D_m\ge\eta^\alpha$. Substitute into (M3.5).
\end{proof}

\subsection{A power trial and uniform comparison over all weights}
\label{sec:weighted-comparison}
For $0<a<1$ and $e>a/2$ put
\begin{equation}\label{eq:pair-beta}
 \mathcal Z(a,e):=\frac{2e^2B(e,1-a)}{2e-a}.
\end{equation}
If $X,Y$ are independent with density $ex^{e-1}$ on $[0,1]$,
splitting the square into two triangles gives
$\mathcal Z(a,e)=\E|X-Y|^{-a}$. In particular
$Z_m=\mathcal Z(s,\eta)$, where $s=m\alpha$ and $\eta=1-d$.

\begin{lemma}[Explicit power trials]\label{lem:power-trial}
Let $p=m/(m-1)$. For every $\tau\ge0$ define
\[
 f_\tau(x)=(1+p\tau)^{1/p}x^{\eta\tau},\qquad
 H_\tau=\frac{(1+p\tau)^{2/p}}{(1+\tau)^2}
           \mathcal Z(\alpha,\eta(1+\tau)).
\]
Then $\|f_\tau\|_{L^p(\mu_d)}=1$ and $D_m\ge H_\tau$.
In particular, with
\begin{equation}\label{eq:power-choice}
 \tau_*:=\frac{\alpha(m-1)}{2-s},\qquad
 L^{(0)}:=\frac{\mathcal Z(\alpha,\eta)^{m/2}}{\sqrt{m!Z_m}},
 \qquad L^{(P)}:=\frac{H_{\tau_*}^{m/2}}{\sqrt{m!Z_m}},
\end{equation}
the following are explicit lower and upper bounds for $\Lambda_m$:
\begin{equation}\label{eq:weighted-sandwich}
 \underline L:=\max\{L_{m,\alpha,d},L^{(0)},L^{(P)}\},
 \qquad
 \overline U:=\min\{(m!)^{-1/2},\widetilde U_{m,\alpha,d}\}.
\end{equation}
\end{lemma}
\begin{proof}
Direct integration gives
\[
 \int_0^1 f_\tau(x)^p\eta x^{\eta-1}\dd x
 =(1+p\tau)\frac{\eta}{\eta(1+p\tau)}=1.
\]
Multiplication by $x^{\eta\tau}$ changes the power density from
exponent $\eta$ to exponent $\eta(1+\tau)$, with mass factor
$(1+\tau)^{-1}$. Doing this in both variables gives the quadratic
value $H_\tau$. Apply Proposition~M3.1 and the
previous bounds. All beta arguments are positive in the stated region.
\end{proof}

\begin{proof}[Proof of Theorem~\ref{main:thm:weighted-global}]
We first consider $1/2<\eta\le1$. For independent $X,Y\sim\mu_d$,
write
\[
 \delta(\eta):=\E|X-Y|
 =\frac{2\eta}{(\eta+1)(2\eta+1)}\le\frac38.
\]
The last inequality follows from
$6\eta^2-7\eta+3>0$. Jensen's inequality gives
$\mathcal Z(\alpha,\eta)\ge\delta(\eta)^{-\alpha}$.
With $t=s/(2\eta)$, admissibility implies
$t<s/(1+s)$. Therefore
\[
 \frac{\overline U}{\underline L}
 \le\frac{\widetilde U}{L^{(0)}}
 \le\frac{2^t}{1-t}\delta(\eta)^{s/2}
 \le q_+(s):=(1+s)2^{s/(1+s)}(3/8)^{s/2}.
\]
For $0<s<1$,
\[
 (\log q_+)'(s)=\frac1{1+s}
 +\frac{\log2}{(1+s)^2}+\frac12\log(3/8)
 \ge\frac12+\frac14\log2+\frac12\log(3/8)>0.
\]
Thus $q_+(s)\le q_+(1)=\sqrt3<C_*$.

Now suppose $\eta\ge1$. Gautschi's inequality
\cite[Eq.~5.6.4]{DLMF} gives, for $e>\alpha/2$,
\[
 \mathcal Z(\alpha,e)
 =\frac{e}{e-\alpha/2}
   \frac{\Gamma(e+1)\Gamma(1-\alpha)}{\Gamma(e+1-\alpha)}
 \ge e^\alpha\Gamma(1-\alpha).
\]
Consequently
\[
 H_\tau\ge\eta^\alpha\Gamma(1-\alpha)F(\tau),
 \qquad
 F(\tau):=\frac{(1+p\tau)^{2/p}}{(1+\tau)^{2-\alpha}}.
\]
Its logarithmic derivative is
$2/(1+p\tau)-(2-\alpha)/(1+\tau)$. It changes sign from
positive to negative exactly at $\tau_*$ in \eqref{eq:power-choice}.
Thus $\tau_*$ maximizes the explicit lower envelope $F$.
The identities
\[
 1+p\tau_*=(1-s/2)^{-1},\qquad
 1+\tau_*=(1-\alpha/2)/(1-s/2)
\]
then give the sharper parameter-dependent comparison
\begin{equation}\label{eq:weighted-ratio-refined}
 \frac{\overline U}{\underline L}\le
 \frac{\widetilde U}{L^{(P)}}\le R_m(s):=
 \frac{2^{s/2}(1-\alpha/2)^{m-s/2}}
 {(1-s/2)^{2-s/2}\Gamma(1-\alpha)^{m/2}}.
\end{equation}
Using the logarithmic gamma series, equivalently
\cite[Eq.~5.7.3]{DLMF},
\[
 \log\Gamma(1-\alpha)
 =\gamma_E\alpha+\sum_{k=2}^{\infty}\frac{\zeta(k)}k\alpha^k
 \ge\gamma_E\alpha+\frac12\alpha^2,
\]
and $\log(1-\alpha/2)\le-\alpha/2$, we obtain
\[
 (m-s/2)\log(1-\alpha/2)
 -\frac m2\log\Gamma(1-\alpha)
 \le-\frac{1+\gamma_E}{2}s.
\]
The terms $s^2/(4m)$ cancel. It follows that
\[
 R_m(s)\le q_-(s):=
 \frac{2^{s/2}\exp[-(1+\gamma_E)s/2]}{(1-s/2)^{2-s/2}}.
\]
For $v=1-s/2\in(1/2,1)$ its logarithmic derivative is
\[
 (\log q_-)'(s)=\frac12\{\log(2v)-\gamma_E+v^{-1}\}>0.
\]
Thus $q_-(s)\le q_-(1)=C_*$. This proves the ratio bound.
Corollary~\ref{cor:relative} proves the harmonic-mean error statement.
At $d=0$, $\overline U=U_m$ and
$\underline L\ge L_{m,\alpha,0}$, so the earlier guarantee is retained.
\end{proof}

\begin{corollary}[Behavior under concentration of the weight]
\label{cor:power-concentration}
Fix $m\ge2$ and $0<\alpha<1/m$. As $\eta=1-d\to\infty$,
\[
 L^{(0)}\longrightarrow
 \frac{\Gamma(1-\alpha)^{m/2}}{\sqrt{m!\Gamma(1-s)}},\qquad
 L^{(P)}\longrightarrow
 \frac{[\Gamma(1-\alpha)F(\tau_*)]^{m/2}}
      {\sqrt{m!\Gamma(1-s)}}>0.
\]
The second limit is a lower bound for
$\liminf_{d\to-\infty}\Lambda_m(\alpha,d)$ and exceeds the first.
\end{corollary}
\begin{proof}
The gamma-ratio asymptotic \cite[Section 5.11(iii)]{DLMF} gives
$\mathcal Z(a,e)\sim\Gamma(1-a)e^a$ for fixed $0<a<1$.
Substitute in $H_{\tau_*}$ and $Z_m$. Since $F'(0)>0$ and
$\tau_*>0$ is its maximum, $F(\tau_*)>F(0)=1$.
\end{proof}
The new trials remedy the degeneration in
Proposition~\ref{prop:negative-d}. Section~M8 proves convergence
of $\Lambda_m$ with a quadratic rate.